\documentclass[12pt]{amsart}

\usepackage{verbatim}
\usepackage{amssymb}
\usepackage{epsfig}
\usepackage[all]{xy}

\newtheorem{theorem}{Theorem}[section]
\newtheorem{corollary}[theorem]{Corollary}
\newtheorem{lemma}[theorem]{Lemma}
\newtheorem{proposition}[theorem]{Proposition}

\newtheorem{definition/theorem}[theorem]{Definition/Theorem}

\theoremstyle{definition} \newtheorem{definition}[theorem]{Definition}
\newtheorem{remark}[theorem]{Remark}
\newtheorem{example}[theorem]{Example}

\newtheorem{question}[theorem]{Question}

\newtheorem{notation}[theorem]{Notation}

\newcommand{\K}{\Bbbk}
\newcommand{\MOn}{\overline{M}_{0,n}}
\newcommand{\git}{\ensuremath{\operatorname{/\!\!/}}}

\DeclareMathOperator{\intt}{int}

\DeclareMathOperator{\trop}{Trop}
\DeclareMathOperator{\inn}{in}
\DeclareMathOperator{\Pic}{Pic}
\DeclareMathOperator{\DivCl}{Cl}
\DeclareMathOperator{\pos}{pos}

\DeclareMathOperator{\im}{im}
\DeclareMathOperator{\rank}{rank}
\DeclareMathOperator{\Hom}{Hom}

\DeclareMathOperator{\Spec}{Spec}
\DeclareMathOperator{\Proj}{Proj}

\DeclareMathOperator{\Hilb}{Hilb}
\DeclareMathOperator{\Chow}{Chow}
\DeclareMathOperator{\Cox}{Cox}

\DeclareMathOperator{\val}{val}

\begin{document}

\title{Equations for Chow and Hilbert Quotients}

\author{Angela Gibney}

\address{Department of Mathematics\\
Department of Mathematics\\
University of Georgia\\
Athens, GA 30602\\
USA
}

\email{agibney@math.uga.edu}

\author{Diane Maclagan}
\address{Mathematics Institute\\
Zeeman Building\\
University of Warwick\\
Coventry CV4 7AL\\
United Kingdom}

\email{D.Maclagan@warwick.ac.uk}

\begin{abstract}
 We give explicit equations for the Chow and Hilbert quotients of a
  projective scheme $X$ by the action of an algebraic torus $T$ in an
  auxiliary toric variety. As a consequence we provide GIT
  descriptions of these canonical quotients, and obtain other GIT
  quotients of $X$ by variation of GIT quotient.  We apply these
  results to find equations for the moduli space $\MOn$ of stable
  genus zero $n$-pointed curves as a subvariety of a smooth 
  toric variety defined via tropical methods.
\end{abstract}

\date{\today}



\maketitle

\section{Introduction}

When a reductive group $G$ acts linearly on a projective scheme $X$, a
fundamental problem is to describe a good notion of a quotient $X/G$.
This question frequently arises in the construction and
compactification of moduli spaces.  In many situations, there is an
open subset $U \subset X$ on which $G$ acts freely, such that a scheme
$U/G$ exists as a geometric quotient.  Constructing the quotient $X/G$
is thus choosing a good compactification of $U/G$.  One way to
compactify is by forming the Chow quotient $X\git^{Ch} G$ or Hilbert
quotient $X \git^H G$ of $X$ by $G$ (see~\cite{Kapranov}).  These
quotients are taken to be the closure of $U/G$ in an appropriate Chow
variety or Hilbert scheme.  They are natural
canonical quotients with proper birational maps to any GIT quotient.
See also~\cite{CQ2}, \cite{YiHu}.

In this paper we treat the case where $G=T^d$ is a $d$-dimensional
algebraic torus acting equivariantly on a subscheme $X$ of $\mathbb
P^m$.  Given the ideal $I$ of $X$ as a subscheme of $\mathbb P^m$, we
describe equations for $X \git^{Ch} T^d$ and $X\git^{H} T^d$ in the
Cox rings of toric subvarieties of the Chow and Hilbert quotients of
$\mathbb P^m$.

  As a first application of our results, we give GIT constructions of
$X\git^{Ch}T^d$ and $X \git^{H}T^d$ and we prove that all GIT
quotients of $X$ by $T^d$ can be obtained from the Chow and Hilbert
quotients by variation of GIT.

As a second  application we study the action of an $(n-1)$-dimensional torus
$T^{n-1}$ on the Grassmannian $G(2,n)$.  Here we can take $U$
to be the points with nonvanishing Pl\"ucker coordinates, and the
quotient $U/T^{n-1}$ is the moduli space $M_{0,n}$ of smooth
$n$-pointed genus zero curves.  In this case the desired
compactification is the celebrated moduli space $\MOn$ of stable
$n$-pointed genus zero curves.  Kapranov has shown~\cite{Kapranov}
that $\MOn$ is isomorphic to both the Chow and Hilbert quotients of
$G(2,n)$ by the $T^{n-1}$-action.  We give explicit equations for
$\MOn$ as a subvariety of a smooth toric variety $X_{\Delta}$ whose
fan is the well-studied space of phylogenetic trees.  We show that the
equations for $\overline{M}_{0,n}$ in the Cox ring $S$ of $X_{\Delta}$
are generated by the Pl\"ucker relations homogenized with respect to
the grading of $S$.

We now describe our results in more detail.  The notation $X
\git^{\star} T^d$ is used to refer to either the Chow or the Hilbert
quotient.  We assume that no irreducible component of $X$ lies in any
coordinate subspace.  This means that $X \git^{\star} T^d$ is a
subscheme of $\mathbb P^m \git^{\star} T^d$.  The quotient $\mathbb
P^m \git^{\star} T^d$ is a not-necessarily-normal toric
variety~\cite{KSZ1} whose normalization we denote by
$X_{\Sigma^{\star}}$.  By $X \git^{\star}_n T^d$ we mean the pullback
of $X \git^{\star} T^d$ to $X_{\Sigma^{\star}}$.  Our main theorem, in
slightly simplified form, is the following.  This is proved in
Theorems~\ref{t:Eqmainthm} and \ref{t:VGIT} and
Proposition~\ref{p:GIT}.

\begin{theorem} \label{t:mainthm} Let $T^d \cong (\K^{\times})^d$ act on $\mathbb
  P^m$ and let $X \subset \mathbb P^m$ be a $T^d$-equivariant
  subscheme with corresponding ideal $I(X) = \langle f_1,\dots,f_g
  \rangle \subset \K[x_0,\dots,x_m]$.  Let $X_{\Sigma} \subset 
  X_{\Sigma^{\star}}$ be any toric subvariety with $X \subseteq
  X_{\Sigma} \subseteq X_{\Sigma^{\star}}$.
\begin{enumerate}
\item {\bf{(Equations)}} The ideal $I$ of the Hilbert or Chow quotient
$X \git_{n}^{{\star}} T^d$ in the Cox ring $S=\K[y_1, \ldots y_{r}]$
of ${X}_{\Sigma}$ can be computed effectively.  Explicitly, $I$ is
obtained by considering the $f_i$ as polynomials in
$y_1,\ldots,y_{m+1}$, homogenizing them with respect to the
$\DivCl(X_{\Sigma})$-grading of $S$, and then saturating the result by
the product of all the variables in $S$.

\item {\bf{(GIT)}} There is a GIT construction of the Chow and Hilbert
quotients of $X$, and these are related to the GIT quotients of $X$ by
variation of GIT quotient.  This gives equations for all quotients
in suitable projective embeddings.  Let $H=\Hom(\DivCl(X_{\Sigma}),
\K^{\times})$.  There is a nonzero cone $\mathcal{G} \subset
\DivCl(X_{\Sigma})\otimes \mathbb R$ for which $X \git_{n}^{{\star}}
T^d $ is the geometric invariant theory quotient
$$X\git_{n}^{{\star}} T^d = Z(I) \git_{\alpha} H,$$ for any rational
$\alpha \in \mathrm{relint}(\mathcal{G})$, where $Z(I)$ is the subscheme of $\mathbb A^r$
defined by $I$.  For any GIT quotient $X \git_{\beta} T^d$ of $X$,
there are choices of $\alpha$ outside $\mathcal{G}$ for which $Z(I)
\git_{\alpha} H = X \git_{\beta} T^d$.
\end{enumerate}

\end{theorem} 

A more precise formulation of the homogenization is given in
Theorem~\ref{t:Eqmainthm} and Remark~\ref{r:homogenization}.  We
explain in Corollary~\ref{c:alpha} how each choice of $\alpha \in
\mathrm{relint}(\mathcal{G})$ gives an embedding of $X \git^{\star}_{n} T^d$ into some
projective space.

 We use tropical algebraic geometry in the spirit of
Tevelev~\cite{Tevelev} to embed $\overline{M}_{0,n}$ in a smooth toric
variety $X_{\Delta}$.  The combinatorial data describing $\Delta$ and
the simple equations for $\overline{M}_{0,n}$ in the Cox ring of
$X_{\Delta}$ are described in the following theorem.  Let
$[n]=\{1,\dots,n\}$ and set $\mathcal I = \{ I \subset [n]: 1 \in I,
|I|, |[n] \setminus I| \geq 2 \}$.  The set $\mathcal I$ indexes the
boundary divisors of $\MOn$.

\begin{theorem} \label{t:M0nmainthm} Let $\Delta$ be the fan in
  $\mathbb R^{{n \choose 2}-n}$ described in Section~\ref{s:MOn} (the
  space of phylogenetic trees).  The rays of $\Delta$ are indexed by
  the set $\mathcal I$.

\begin{enumerate}
\item {\bf{(Equations)}} Equations for $\MOn$ in
the Cox ring $S=\K[x_I : I \in \mathcal I]$ of $X_{\Delta}$ are
obtained by homogenizing the Pl\"ucker relations with respect to the
grading of $S$ and then saturating by the product of the variables of
$S$.  Specifically, the ideal is
$$I_{\MOn}=\left( \left \langle \prod_{i,j \in I, k,l \not \in I} x_I
- \prod_{i,k \in I, j,l \not \in I} x_I + \prod_{i,l \in I, j,k \not
  \in I} x_I \right \rangle : (\prod_{I} x_I)^{\infty}\right),$$ where
the generating set runs over all $\{i,j,k,l\}$ with $1 \leq i <j<k<l
\leq n$, and $x_I=x_{[n]\setminus I}$ if $1 \not \in I$.

 \item {\bf{(GIT)}} There is a nonzero cone $\mathcal G \subset
\DivCl(X_{\Delta}) \otimes \mathbb R \cong \Pic(\MOn) \otimes \mathbb
R$ for which for rational $\alpha \in \intt(\mathcal G)$ we have the GIT
construction of $\MOn$ as
$$\MOn = Z(I_{\MOn}) \git_{\alpha} H,$$ where $Z(I_{\MOn}) \subset
\mathbb A^{|\mathcal I|}$ is the affine subscheme defined by $I_{\MOn}$, and $H$
is the torus $\Hom(\DivCl(X_{\Delta}), \K^{\times}) \cong (\K^{\times})^{|\mathcal I|-{n
\choose 2}+n}$.

\item {\bf{(VGIT)}} Given $\beta \in \mathbb Z^n$ there is $\alpha \in \mathbb
Z^{|\mathcal I|-{n \choose 2}+n}$ for which
$$Z(I_{\MOn}) \git_{\alpha} H = G(2,n) \git_{\beta} T^{n-1},$$ so all
GIT quotients of $G(2,n)$ by $T^{n-1}$ can be obtained from $\MOn$ by
variation of GIT.
\end{enumerate}
\end{theorem}

Part (3) of Theorem~\ref{t:M0nmainthm} relates to the work of
\cite{HMSV} where GIT quotients of $G(2,n)$ by $T^{n-1}$, or
equivalently of $(\mathbb P^1)^n$ by $\mathrm{Aut}(\mathbb P^1)$, are
studied.

In {\em Equations for $\MOn$}~\cite{KeelTevelevEqtns}, Keel and
Tevelev study the image of the particular embedding of $\MOn$ into a
product of projective spaces given by the complete linear series of
the very ample divisor $\kappa = K_{\MOn} + \sum_{I \in \mathcal I}
\delta_I$.  Theorem~\ref{t:M0nmainthm} concerns projective embeddings
of $\MOn$ corresponding to a full-dimensional subcone of the nef cone
of $\MOn$, including that given by $\kappa$.

A key idea of this paper is to work in the Cox ring of sufficiently
large toric subvarieties of $X_{\Sigma^{\star}}$.
This often allows one to give equations in fewer variables.  Also, a
 truly concrete description of $X_{\Sigma^{\star}}$ may be cumbersome
 or impossible, as in the case of $\MOn$, but a sufficiently large
 toric subvariety such as $X_{\Delta}$ can often be obtained.

We now summarize the structure of the paper.
Section~\ref{s:torictools} contains some tools from toric geometry
that will be useful in the rest of the paper.  The first part of
Theorem~\ref{t:mainthm} is proved in Section~\ref{s:Equations}, while
the GIT results are proved in Section~\ref{s:GIT}.  In
Section~\ref{s:MOn} we explicitly describe the toric variety
$X_{\Delta}$ that contains $\overline{M}_{0,n}$.
Theorem~\ref{t:M0nmainthm} is proved in Sections~\ref{s:MOnProof} and
\ref{s:VGIT}.  We end the paper with some natural questions about
$\MOn$ arising from this work.

{\bf{Acknowledgements:}} We thank Klaus Altmann for helpful
discussions about Chow quotients.  We also thank Sean Keel and Jenia Tevelev
for useful conversations.  The authors were partially
supported by NSF grants DMS-0509319 (Gibney) and DMS-0500386
(Maclagan).

\section{Toric tools} \label{s:torictools}

In this section we develop some tools to work with toric varieties
that will be used in our applications to Chow and Hilbert quotients.
We generally follow the notational conventions for toric varieties
of~\cite{Fulton}, with the exception that we do not always require
normality.  Throughout $\K$ is an algebraically closed field, and
$\K^{\times} = \K \setminus \{0\}$.  We denote by $T^d$ an algebraic
torus isomorphic to $(\K^{\times})^d$.  If $I$ is an ideal in
$\K[x_0,\dots,x_m]$ then $Z(I)$ is the corresponding subscheme of
either $\mathbb A^{m+1}$ or $\mathbb P^m$ depending on the context.

\subsection{Producing equations for quotients of subvarieties of tori}
\label{ss:equationsforquotients}

We first describe how to obtain equations for the quotient of a
   subvariety of a torus by a subtorus.  Let $Y$ be a subscheme of a
   torus $T^{m}$ that is equivariant under a faithful action of
   $T^{d}$ on $T^{m}$ given by $(t \cdot x)_j = (\prod_{i=1}^{d}
   t_i^{a_{ij}}) x_j$, and let $I(Y) \subset \K[x_1^{\pm 1}, \dots,
   x_m^{\pm 1}]$ be the ideal of $Y$.  Write $A$ for the $d \times m$
   matrix with $ij$th entry $a_{ij}$.  Let $D$ be a $(m-d) \times m$
   matrix of rank $m-d$  whose rows generate the integer kernel of $A$, so
   $AD^T=0$.
 The matrix $D$ is a {\em Gale dual} for the $d \times m$ matrix
$A=(a_{ij})$ (see \cite[Chapter 6]{Ziegler}).

\begin{proposition} \label{p:quotienteqtns} 
  Let $\phi \colon
  \K[z_1^{\pm 1},\dots, z_{m-d}^{\pm 1}] \rightarrow \K[x_1^{\pm
  1},\dots,x_m^{\pm 1}]$ be given by $\phi(z_i)=\prod_{j=1}^{m}
  x_j^{D_{ij}}$.  Then the ideal of $Y/T^{d}$ in the coordinate ring
  $\K[z_1^{\pm 1},\dots,z_{m-d}^{\pm 1}]$ of $T^{m}/T^{d}$ is
  given by $\phi^{-1}(I(Y))$.  This is generated by polynomials $g_1, \dots, g_s$ for which $I(Y)=\langle \phi(g_1),\dots,\phi(g_s) \rangle$.
\end{proposition}

\begin{proof} 
  The coordinate ring of the quotient $Y/T^{d}$ is by definition the
  ring of invariants of $\K[x_1^{\pm 1},\dots, x_m^{\pm 1}]/I(Y)$ under
  the induced action of $T^{d}$. The $T^{d}$ action on $T^{m}$
  gives a $\mathbb Z^{d}$-grading of $\K[x_i^{\pm 1}]$ by setting
  $\deg(x_i)=\mathbf{a}_i$, where $\mathbf{a}_i$ is the $i$th column
  of the matrix $A$.  Since $T^{d}$ acts equivariantly on $Y$, the
  ideal $I(Y)$ is homogeneous with respect to this grading, so
  $\K[x_i^{\pm 1}]/I(Y)$ is also $\mathbb Z^{d}$-graded.  The ring
  of invariants is precisely the degree-zero part of this ring.

  To prove that this is isomorphic to $k[z_1^{\pm
  1},\dots,z_{m-d}^{\pm 1}]/\phi^{-1}(I(Y))$, we first define an
  automorphism of the torus $T^{m}$ so that $T^{d}$ is mapped to the
  subtorus with first $m-d$ coordinates equal to one.  Choose any $m
  \times m$ integer matrix $U$ with determinant one whose first $d$
  rows consist of the matrix $D$.  This is possible because by the
  definition of $D$ the cokernel $\mathbb Z^{m}/\im(D^T)$ is
  torsion-free, so $\mathbb Z^{m} \cong \im(D^T) \oplus \mathbb
  Z^{d}$.  Then the map $\widetilde{\phi} \colon \K[z_{1}^{\pm
  1},\dots,z_{m}^{\pm 1}] \rightarrow \K[x_1^{\pm 1}, \dots, x_m^{\pm
  1}]$ defined by $\widetilde{\phi}(z_i)=\prod_{j=1}^{m} x_j^{U_{ij}}$
  determines an automorphism of the torus $T^{m}$.  Note that the map
  $\phi$ is $\widetilde{\phi}$ restricted to the ring $\K[z_1^{\pm
  1},\dots,z_{m-d}^{\pm 1}]$.

  The ring $\K[z_1^{\pm 1},\dots,z_{m}^{\pm 1}]$ gets an induced
  $\mathbb Z^{d}$-grading from the grading on $\K[x_j^{\pm 1}]$ by
  setting $\deg(z_i)=\sum_{j=0}^m U_{ij}\mathbf{a}_j$, which is the
  $i$th column of $AU^T$.  Since the first $m-d$ rows of $U$ are the
  rows of $D$, and $AD^T=0$, we thus have $\deg(z_i)=0 \in \mathbb
  Z^{d}$ for $1 \leq i \leq m-d$.  The degrees of the $d$
  variables $z_{m-d+1},\dots,z_{m}$ are linearly independent, since
  $\rank(AU^T)$ is $d$.  This means that the degree zero part of
  $\K[z_1^{\pm 1},\dots, z_{m}^{\pm 1}]$ is $\K[z_1^{\pm
  1},\dots,z_{m-d}^{\pm 1}]$, which proves that the coordinate ring of
  $Y/T^{d}$ is given by $\K[z_1^{\pm 1},\dots,z_{m-d}^{\pm 1}]/J$,
  where $J = \widetilde{\phi}^{-1}(I(Y)) \cap \K[z_1^{\pm 1},\dots,
  z_{m-d}^{\pm 1}]$.  The result then follows since $J=
  \phi^{-1}(I(Y))$.  The statement about generators follows from the
  fact that $\phi$ is injective, since $\widetilde{\phi}$ is an
  isomorphism.
\end{proof}

\begin{example} \label{e:torusquotient}
Let $Y$ be the subscheme of $T^{10}$ defined by the ideal $I=\langle
x_{12}x_{34}-x_{13}x_{24}+x_{14}x_{23},
x_{12}x_{35}-x_{13}x_{25}+x_{15}x_{23},
x_{12}x_{45}-x_{14}x_{25}+x_{15}x_{24},
x_{13}x_{45}-x_{14}x_{35}+x_{15}x_{34},
x_{23}x_{45}-x_{24}x_{35}+x_{25}x_{34} \rangle \subseteq
\K[x_{ij}^{\pm 1} : 1 \leq i < j \leq 5]$.
 This is the intersection with the
torus of $\mathbb A^{10}$ of the affine cone over the Grassmannian
$G(2,5)$ in its Pl\"ucker embedding into $\mathbb P^9$.  The torus
$T^5$ acts equivariantly on $Y$ by $t \cdot x_{ij} = t_{i}t_j x_{ij}$, giving rise to matrices 
\renewcommand{\arraystretch}{0.8}
\renewcommand{\arraycolsep}{2pt}
\begin{equation*} 
A = \left( \text{\footnotesize $\begin{array}{llllllllll}  
1 & 1 & 1 & 1& 0 & 0& 0 & 0 & 0 & 0 \\  
1 & 0 & 0 & 0& 1 & 1& 1 & 0 & 0 & 0 \\  
0 & 1 & 0 & 0& 1 & 0& 0 & 1 & 1 & 0 \\  
0 & 0 & 1 & 0& 0 & 1& 0 & 1 & 0 & 1 \\  
0 & 0 & 0 & 1& 0 & 0& 1 & 0 & 1 & 1 \\  
             \end{array}$} \right),  \, \, \,
D= \left( \text{\footnotesize $\begin{array}{rrrrrrrrrr}
0 & 1 & -1 & 0 &-1& 1 & 0 & 0 & 0 & 0 \\
0 & 1 & 0 & -1 &-1& 0 & 1 & 0 & 0 & 0 \\
1 & 0 & -1 & 0 &-1& 0 & 0 & 1& 0 & 0 \\
1 & 0 & 0 & -1 &-1& 0 & 0 & 0 & 1 & 0 \\
1 & 1 & -1 & -1 &-1& 0 & 0 & 0 & 0 & 1 \\
            \end{array}$} \right),
\end{equation*}
where the columns are ordered $\{12, 13, \dots, 35,45 \}$.  The map
$\phi \colon \K[z_1^{\pm 1}, \dots, z_5^{\pm 1}] \rightarrow
\K[x_{ij}^{\pm 1} : 1 \leq i < j \leq 5 ]$ is given by $\phi(z_1) =
x_{13}x_{24}/x_{14}x_{23}$, $\phi(z_2) = x_{13}x_{25}/x_{15}x_{23}$,
$\phi(z_3) = x_{12}x_{34}/x_{14}x_{23}$, $\phi(z_4) =
x_{12}x_{35}/x_{15}x_{23}$, and $\phi(z_5)
=x_{12}x_{13}x_{45}/x_{14}x_{15}x_{23}$.

The ideal $\phi^{-1}(I)$ is then $\langle z_3 - z_1 + 1, z_4 - z_2 +1,
z_5-z_2+z_1, z_5-z_4+z_3, z_5-z_1z_4+z_2z_3 \rangle = \langle
z_3-z_1+1, z_4-z_2+1, z_5-z_2+z_1 \rangle \subset \K[z_1^{\pm
1},\dots,z_5^{\pm 1}]$.  For this it is essential that we work in the
Laurent polynomial ring; for example, $\phi(z_5-z_2+z_1) =
x_{13}/(x_{14}x_{15}x_{23})(x_{12}x_{45}-x_{14}x_{25}+x_{15}x_{24})$.
The variety of $\phi^{-1}(I)$ is the moduli space $M_{0,5}$.  Note
that this shows that $M_{0,5}$ is a complete intersection in $T^5$,
cut out by three linear equations.  This example is continued in
Example~\ref{e:M0nasChow} and Sections~\ref{s:MOn}, \ref{s:MOnProof},
and \ref{s:VGIT}.

\end{example}

\subsection{Producing equations for closures in toric varieties}

The next proposition describes how to find the ideal of the closure of
a subvariety of a torus in a toric variety.  We use the notation
$T^{m-d}$ for ease of connection with the rest of this section, but there
is no requirement that this torus be obtained as a quotient.  

Recall that the Cox ring of a normal toric variety $X_{\Sigma}$ (see
\cite{Cox} and \cite{Mustata}) is the polynomial ring $S=
\K[y_1,\dots,y_r]$, where $r=|\Sigma(1)|$ is the number of rays of
$\Sigma$.  It is graded by the divisor class group of $X_{\Sigma}$, so
that $\deg(y_i)=[D_i]$, where $[D_i] \in \DivCl(X_{\Sigma})$ is the
class of the torus-invariant divisor $D_i$ associated to the $i$th ray
$\rho_i$ of $\Sigma$.  An ideal $I \subset S$ determines an ideal
sheaf $\widetilde{I}$ on $X_{\Sigma}$, and thus a closed subscheme of
$X_{\Sigma}$, and conversely, every ideal sheaf on $X_{\Sigma}$ is of
the form $\widetilde{I}$ for some ideal $I$ of $S$ (Musta\c{t}\u{a}'s
Theorem 1.1 removes the need for the simplicial hypothesis in Cox's
Theorem 3.7).  The sheaf $\widetilde{I}$ is given on an affine chart
$U_{\sigma}$ of $X_{\Sigma}$ by $I_{\sigma}=(IS_{\prod_{i \not \in
\sigma} y_i})_{\mathbf{0}}$.  The correspondence between ideals in $S$ and closed
subschemes of $X_{\Sigma}$ is not bijective, but for any closed
subscheme $Z \subset X_{\Sigma}$ there is a largest ideal $I(Z)
\subset S$ with $\widetilde{I(Z)} = \mathcal I_Z$.

Recall also that if $I$ is an ideal in a ring $R$ and $y \in R$, then
$(I: y^{\infty}) = \{ r \in R : ry^k \in I \text{ for some } k >0 \}$.
Geometrically this removes irreducible components supported on the
variety of $y$.

\begin{proposition} \label{p:closureeqtns}
Let $X_{\Sigma}$ be an $(m-d)$-dimensional toric variety with Cox ring
$S=\K[y_1,..,y_r]$.  Set $y=\prod_{i=1}^r y_i$ so that
$$\rho: \K[T^{m-d}]=\K[z^{\pm 1}_1, ..., z^{\pm 1}_{m-d}]
\longrightarrow (S_y)_0$$ is the isomorphism given by the inclusion of
the torus $T^{m-d}$ into $X_{\Sigma}$.  If $Y \subset T^{m-d}$ is
given by ideal $I(Y)= \langle f_1,..,f_s \rangle \subset \K[T^{m-d}]$,
then the ideal $I$ for the closure $\overline{Y}$ of $Y$ in
$X_{\Sigma}$ is $(\rho(I(Y))S_y)\cap S$, which is 
$$I= \left( \langle {\widetilde{\rho(f_i)}: 1 \le i \le s} \rangle :
y^{\infty}\right),$$ where $\widetilde{\rho(f_i)}$ is obtained by
clearing the denominator of $\rho(f_i)$.
\end{proposition}

\begin{proof}
  Let $J=\rho(I(Y)) \subset (S_y)_0$.  The closure $\overline{Y}$ of
  $Y$ in $X_{\Sigma}$ is the smallest closed subscheme of $X_{\Sigma}$
  containing $Y$.  Since the torus $T^{m-d}$ is the affine toric
  variety corresponding to the cone $\sigma$ consisting of just the
  origin in $\Sigma$, for this $\sigma$ we have $I_{\sigma} =
  (IS_y)_{\mathbf{0}}$ for any ideal $I \subseteq S$.  As the
  correspondence between subschemes $Z$ of $X_{\Sigma}$ and ideals
  $I(Z)$ of $S$ is inclusion reversing, we have that $I(\overline{Y})$
  is the largest ideal $I$ in $S$ with $(I S_y)_0 \subseteq J$ for
  which $I=I(Z)$ for some subscheme $Z$ of $X_{\Sigma}$.  

There is a monomial, and thus a unit, in any degree $a$ for which
$(S_y)_a$ is nonzero, so if $K$ is an ideal in $S_y$, then $K_0S_y=K$,
and $(K \cap S)S_y =K$.  Thus $((J S_{y}\cap S)S_{y})_0 = (J S_y)_0 =
J$, and if $I$ is any homogeneous ideal in $S$ with $(IS_y)_0 \subset
J$ then $I \subseteq IS_y \cap S \subset JS_y \cap S$.  Thus to show
that $I=JS_y \cap S$ is the ideal of $I(\overline{Y})$, we need only
show that $I$ is of the form $I(Z)$ for some subscheme $Z$ of
$X_{\Sigma}$.  Indeed, let $I'$ be the largest ideal in $S$ with
$\widetilde{I'}=\widetilde{I}$.  Then by construction we have
$(I'S_y)_0 = (IS_y)_0 = J$, so by above $I' \subseteq I$, and thus
$I=I'$.  This shows that $I(\overline{Y})=JS_y \cap S$.

Suppose now that $I(Y)$ is generated by $\{f_1,\dots,f_s \} \subset
\K[z_1^{\pm 1},\dots,z_{m-d}^{\pm 1}]$.  Then $JS_y$ is generated by
$\{ \rho(f_1),\dots,\rho(f_s)\}$.  The denominator of each $\rho(f_i)$
is a monomial, which is a unit in $S_y$, so $JS_y$ is also generated
by the polynomials $\widetilde{\rho(f_i)}$ obtained by clearing the
denominators in the $\rho(f_i)$.  The result follows from the 
observation that if $K$ is an ideal in $S_y$, generated by
$\{g_1,\dots,g_s\} \subset S$, then $K \cap S = (\langle g_1,\dots,g_s
\rangle : y^{\infty})$.  See, for example, Exercise 2.3 of
\cite{Eisenbud}.
\end{proof}

\subsection{Sufficiently large subvarieties via tropical geometry}

\label{ss:tropical}

A key idea of this paper is to work in toric varieties whose fan has few cones.
This is made precise in the following definition.

\begin{definition}\label{d:sufflarge}
A toric subvariety $X_{\Delta}$ of a toric variety $X_{\Sigma}$ with
torus $T^m$ is {\em sufficiently large} with respect to a subvariety
$Y\subseteq X_{\Sigma}$ if the fan $\Delta$ contains all cones of
$\Sigma$ corresponding to $T^m$-orbits of $X_{\Sigma}$ that intersect
$Y$.  
\end{definition}

Tropical geometry provides the tools to compute whether a given toric
subvariety $X_{\Delta} \subseteq X_{\Sigma}$ is sufficiently large for
$Y \subseteq X_{\Sigma}$.  We now review the version we use in this
paper.
Let $Y \subset T^m \cong (\K^{\times})^m$ be a subscheme defined by
the ideal $I=I(Y) \subset S:=\K[x_1^{\pm 1},\dots,x_m^{\pm 1}]$.
Given a vector $w \in \mathbb
R^{m}$ we can compute the leading form $\inn_w(f)$ of a
polynomial $f \in S$, which is the sum of those terms $c_ux^u$ in $f$
with $w \cdot u$ minimal.  The initial ideal $\inn_w(J)$ is $\langle
\inn_w(f) : f \in J \rangle$.  
Let $K$ be any algebraically closed field extension of $\K$ with a
nontrivial valuation $\val : K^{\times} \rightarrow
\mathbb R$ such that $\val(\K)=0$. We denote by $V_K(I)$ the set $\{ u
\in (K^{\times})^{m} : f(u)=0 \text{ for all } f \in I \}$.

\begin{definition/theorem} \label{dt:tropdefn} Let $Y$ be a subvariety of $T^m$. 
 The tropical variety of $Y$, denoted $\trop(Y)$, is the closure in
$\mathbb R^{m}$ of the set $$\{(\val(u_1),\dots, \val(u_m)) \in
\mathbb R^{m}: (u_1,\dots,u_m) \in V_K(J) \}.$$ This equals the set
$$\{ w \in \mathbb R^m : \inn_w(I) \neq \langle 1 \rangle \}.$$ There
is a polyhedral fan $\Sigma$ whose support is $\trop(Y)$.
\end{definition/theorem}

Versions of this result appear in \cite{EKL}, \cite[Theorem
  2.1]{SpeyerSturmfels}, \cite{Draisma}, \cite{JensenMarkwigMarkwig},
\cite{Payne}.  We consider here only the ``constant coefficient''
case, where the coefficients of polynomials generating defining ideal
$I$ live in $\K$, so have valuation zero.  This guarantees there is
the structure of a fan on $\trop(Y)$, rather than a polyhedral
complex.  We note that we follow the conventions for tropical
varieties as tropicalizations of usual varieties as in, for example,
\cite{GathmannTropical} or \cite{SpeyerSturmfels} rather than the more
intrinsic definition used by Mikhalkin~\cite{MikhalkinICM}.  For
readers familiar with these works we emphasize that we follow the
$\min$ convention for the tropical semiring rather than the $\max$
convention of \cite{GathmannTropical}.    

The key result is the following fundamental lemma of Tevelev and its immediate
corollary.

\begin{lemma} \cite[Lemma 2.2]{Tevelev} \label{l:TevelevLemma}
Let $Y$ be a subvariety of the torus $T^m$, and let $\mathcal F$ be an
$l$-dimensional cone in $\mathbb Q^m$ whose rays are spanned by part
of a basis for $\mathbb Z^m \subseteq \mathbb Q^m$.  Let $U_{\mathcal
F} = \mathbb A^l \times (\K^{\times})^{m-l}$ be the corresponding
affine toric variety.  Then the closure $\overline{Y}$ of $Y$ in
$U_{\mathcal F}$ intersects the closed orbit of $U_{\mathcal F}$ if
and only if the interior of the cone $\mathcal F$ intersects the
tropical variety of $Y$.
\end{lemma}

\begin{corollary} \label{c:tevelevcorollary}
Let $Y$ be a subvariety of the torus $T^m$ and let $X_{\Sigma}$ be an
$m$-dimensional toric variety with dense torus $T^m$ and fan $\Sigma
\subseteq \mathbb R^m$.  Let $\trop(Y) \subseteq \mathbb R^m$ be the
tropical variety of $Y \subseteq T^m$.  Then the closure
$\overline{Y}$ of $Y$ in $X_{\Sigma}$ intersects the $T^m$-orbit of
$X_{\Sigma}$ corresponding to a cone $\sigma \subseteq \Sigma$ if and
only if $\trop(Y)$ intersects the interior of $\sigma$.

Thus a toric subvariety $X_{\Delta}$ of $X_{\Sigma}$ with $\Delta$ a
subfan of $\Sigma$ is sufficiently large with respect to
$\overline{Y}$ exactly when $\Delta$ contains every cone of $\Sigma$
whose relative interior intersects the $\trop(Y)$.
\end{corollary}

\begin{proof}
Since $\Sigma$ is not assumed to be smooth, or even simplicial, we
first resolve singularities.  Let $\pi: X_{\Sigma'} \rightarrow
X_{\Sigma}$ be a toric resolution of singularities, so $\Sigma'$ is a
refinement of the fan $\Sigma$, and let $\overline{Y}'$ be the strict
transform of $\overline{Y}$ in $X_{\Sigma'}$, which is the closure of
$Y$ in $X_{\Sigma'}$.  It suffices to prove the corollary for
$\overline{Y}' \subseteq X_{\Sigma'}$, as $\overline{Y}$ intersects
the orbit corresponding to a cone $\sigma \in \Sigma$ if and only if
$\overline{Y}'$ intersects the orbit corresponding to a cone $\sigma'
\in \Sigma'$ with $\sigma' \subseteq \sigma$.  Since the orbit
corresponding to a cone $\sigma' \in \Sigma'$ is the closed orbit of
the corresponding $U_{\sigma'}$, this result now follows from
Lemma~\ref{l:TevelevLemma}.
\end{proof}

\begin{example} \label{e:tropicaleg}
Let $\Sigma_1$ and $\Sigma_2$ be the two complete fans with rays as
pictured in Figure~\ref{f:sufflargeR2}.  If $Y \subset T^2$ is a
subvariety with tropicalization the dotted line shown in both figures,
then the shaded fans define sufficiently large toric subvarieties with
respect to the respective closures of $Y$ in $X_{\Sigma_1}$ and
$X_{\Sigma_2}$.

\begin{figure}
\center{\resizebox{6cm}{!}{\includegraphics{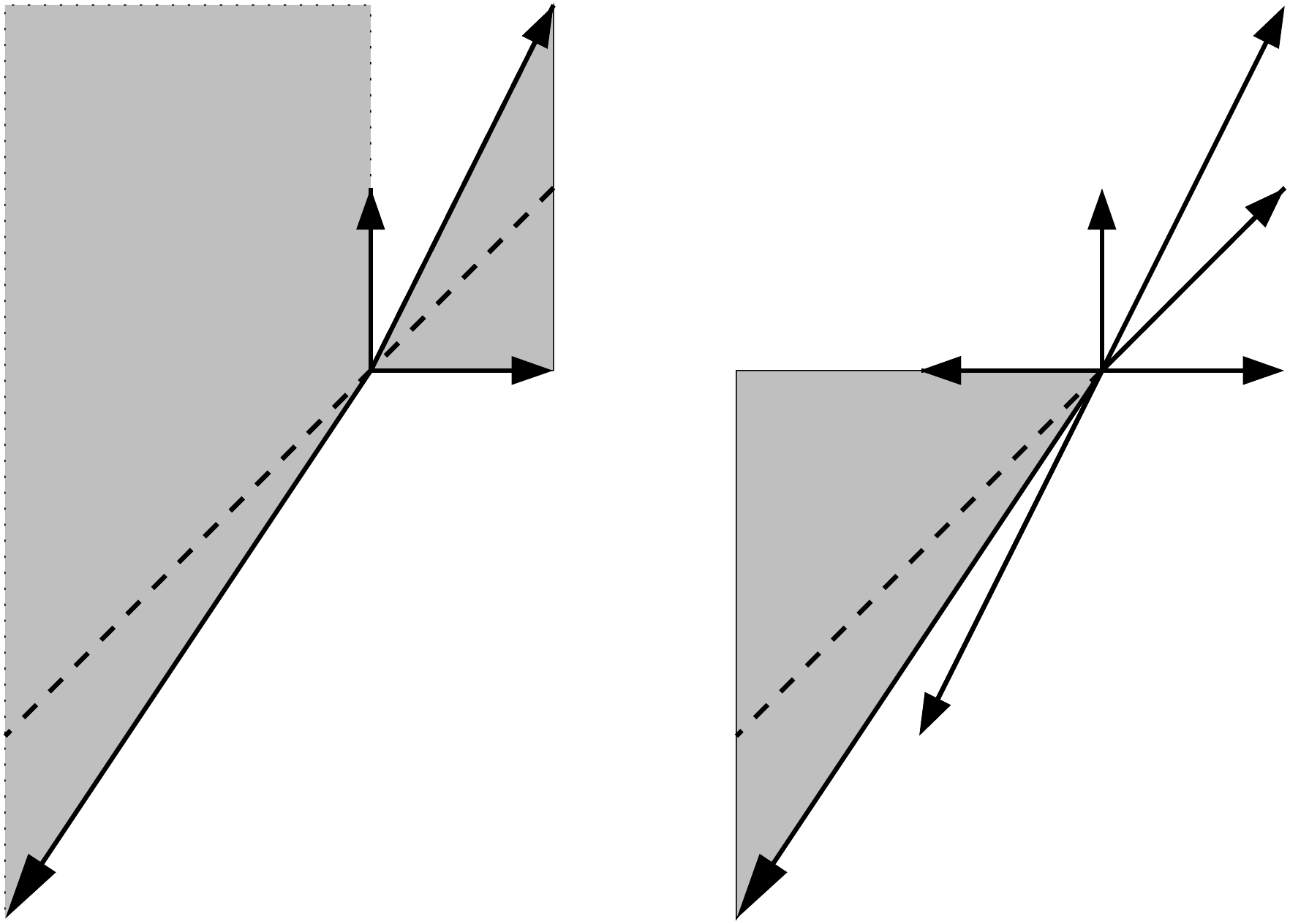}}}
\caption{ \label{f:sufflargeR2} Sufficient toric subvarieties of
$X_{\Sigma}$ with respect to $\overline{Y}$.}
\end{figure}

\end{example}

\section{Equations for Chow and Hilbert quotients}
\label{s:Equations}

Given equations for a $T^d$-equivariant subscheme $X \subset \mathbb
P^m$, we show in this section how to effectively compute generators
for the ideal of the Chow or Hilbert quotient in the Cox ring of a suitably chosen
toric variety.

\subsection{Definition of Chow and Hilbert quotients}

We first recall the definition of the Chow and Hilbert quotients $X
\git^{Ch} T^d$ and $X\git^H T^d$ of a projective variety by the action
of a $d$-dimensional algebraic torus $T^d$.  We assume that $X$ is
equivariantly embedded into a projective space $\mathbb P^m$ with no
irreducible component contained in a coordinate subspace.

Identifying $T^d$ with $(\K^{\times})^{d+1}/(\K^{\times})$ so that points in $T^d$
are equivalence classes $t=[t_0:\ldots:t_d]$, we can write
\begin{equation}
\label{e:Amatrix1}
T^d \times \mathbb P^m \rightarrow \mathbb P^m, \ \
(t,[x_0,\ldots,x_m]) \mapsto [(\prod_{j=0}^d t_j^{a_{j0}}) x_0,\ldots,
(\prod_{j=0}^d t_j^{a_{jm}})x_m],
\end{equation}  and let $A$ be the $(d+1)\times
(m+1)$ matrix with $ij$th entry $a_{ij}$ for $0 \leq i \leq d$ and $0
\leq j \leq m$.
We assume that $a_{0i} =1$ for $0 \leq i \leq m$, and that the
$T^d$-action on $\mathbb P^m$ is faithful, so $A$ has rank $d+1$.  Let
$X_{A} \subset \mathbb P^m$ be the closure of the $T^d$-orbit of $e=[1
\colon \dots \colon 1]\in T^m$.  Then $X_{A}$ is a toric variety with
associated torus $T^d$ and corresponding toric ideal $I_{A}=
\langle x^u - x^v : Au=Av \rangle \subset \K[x_0,\ldots,x_m]$ (see 
\cite[Chapter 4]{GBCP},\cite[Chapter 5]{GKZ}).

We identify $T^m$ with the quotient
$(\K^{\times})^{m+1}/(\K^{\times})$, so a point on $T^m$ is an
equivalence class $s=[s_0: \ldots : s_m]$.  Then $T^m$ acts on
$\mathbb P^m$ by $s \cdot x = [s_0 x_0: \ldots : s_m x_m]$ and every
point in the geometric quotient $T^m/T^d$ corresponds to an orbit of
$T^d$ whose closure in $\mathbb P^m$ is a $d$-dimensional subscheme of
$\mathbb P^m$ having ideal $I_{sA}= \langle s^vx^u - s^ux^v : Au=Av
\rangle \subset \K[x_0,\ldots,x_m]$, with $s \in T^m$.  These orbit
closures all have the same Hilbert polynomial, so define closed points
on the same connected component of $\Hilb(\mathbb P^m)$, and there is
an induced morphism $\phi_H: T^m/T^d \longrightarrow \Hilb(\mathbb
P^m)$.  The Hilbert quotient $X \git^{H}T^d$ is defined to be the
closure in $\Hilb(\mathbb P^m)$ of $\phi_H((X \cap T^m)/T^d)$.  Since
$\mathbb P^m \git^{H} T^d$ is the closure of $\phi_H(T^m/T^d)$ in
$\Hilb(\mathbb P^m)$, $X\git^{H} T^d \subset \mathbb P^m \git^{H}T^d$.

Analogously, there is a  morphism $\phi_{Ch}:  T^m/T^d
\longrightarrow \Chow(\mathbb P^m)$ (see \cite[Sections 1.3--1.4]{Kollar}). The
Chow quotient $X \git^{Ch}T^d$ is defined to be the closure in
$\Chow(\mathbb P^m)$ of $\phi_{Ch}((X \cap T^m)/T^d)$ and
$X\git^{Ch} T^d \subset \mathbb P^m \git^{Ch}T^d$. 

We remark that while the definitions given here appear to depend on
the choice of projective embedding, the Chow and Hilbert quotients of
$X$ are in fact independent of this choice.  See \cite{Kapranov} for a
more intrinsic formulation.  See also~\cite{BBS} for original work on
the Hilbert quotient.  For the most part the Chow and Hilbert
quotients can be treated uniformly, and we use the notation $X
\git^{\star} T^d$ to denote either of $X \git^{Ch} T^d$ or $X \git^{H}
T^d$.

The following example will be developed more in Section~\ref{s:MOn}.

\begin{example} \label{e:M0nasChow}
Consider the action of the $(n-1)$-dimensional
torus $T^{n-1}\cong (\K^{\times})^{n}/\K^{\times}$ on $\mathbb
P_{\mathbb C}^{{n\choose 2}-1}$ given by by
\begin{equation} \label{eqtn:Tnaction}
[t_1:\ldots:t_n]\cdot [\{x_{ij}\}_{1 \le i < j \le n}]=[\{t_i t_j
x_{ij}\}_{1 \le i < j \le n}].
\end{equation}
 Recall that the Pl\"ucker embedding
of the Grassmannian $G(2,n)$ into $\mathbb P^{{n \choose 2}-1}$ is given by taking a
subspace $V$ to $V \wedge V$, or by taking a $2\times n$ matrix
representing a choice of basis for $V$ to its vector of $2\times 2$
minors.  The ideal $I_{2,n}$ of $G(2,n)$ in the homogeneous
coordinate ring of $\mathbb P^{{n\choose 2}-1}$ is generated by the
set of Pl\"ucker equations $\langle p_{ijkl}=x_{ij}x_{kl}-
x_{ik}x_{jl}+ x_{il}x_{jk} : 1 \le i <j<k<l \le n\rangle$ and hence is
$T^{n-1}$-equivariant.  Let $G^0(2,n)$ be the open set inside $G(2,n)$
consisting of those points with nonvanishing Pl\"ucker coordinates,
corresponding to those two-planes that do not pass through the
intersection of any two coordinate hyperplanes.  The torus $T^{n-1}$
acts freely on $G^0(2,n)$ and all orbits are maximal dimensional.  The
moduli space $M_{0,n}$ parametrizing families of smooth $n$-pointed
rational curves is equal to the geometric quotient
$G^0(2,n)/T^{n-1}$.  By \cite[Theorem 4.1.8]{Kapranov}  
$$G(2,n)\git^{Ch} T^{n-1}=G(2,n)\git^{H} T^{n-1} = \MOn.$$
\end{example}

\subsection{Chow and Hilbert quotients of projective spaces} 

\label{ss:ChowHilbPm}

The Chow and the Hilbert quotients of $\mathbb P^m$ by $T^d$ are both
not-necessarily-normal toric varieties~\cite{KSZ1}.  We next describe
the fans $\Sigma^{Ch}$ and $\Sigma^H$ associated to the normalizations
of $\mathbb P^m \git^{Ch} T^d$ and $\mathbb P^m \git^{H} T^d$.  The
fan $\Sigma^{Ch}$ is the {\em secondary fan} of the matrix $A$ given
in (\ref{e:Amatrix1}).  Top-dimensional cones of the secondary fan
correspond to regular triangulations of the vector configuration
determined by the columns of $A$.  These are also indexed by radicals
of initial ideals of $I_A$ by \cite[Theorem 8.3]{GBCP}.  See
\cite{GKZ} for a description of the secondary fan.  The fan $\Sigma^H$
is the {\em saturated Gr\"obner fan}, whose cones are indexed by the
saturation of initial ideals of $I_{A}$ with respect to the irrelevant
ideal $\langle x_0,\dots,x_m\rangle$.  See \cite{BayerMorrison} and
\cite{MoraRobbiano} for the original work on the Gr\"obner fan, and
\cite{GBCP} or \cite{IndiaNotes} for expositions.  When we want to
refer to both fans simultaneously we use the notation
$\Sigma^{\star}$.

We denote by $\overline{N}$ the common lattice of the fans
$\Sigma^{Ch}$ and $\Sigma^H$.  Note that $\overline{N}=N/N'$, where $N
\cong \mathbb Z^{m+1}$ and $N'$ is the integer row
space of the matrix $A$.  The torus $T^{d+1}$
with $T^d=T^{d+1}/\K^{\times}$ is equal to $N' \otimes \K^{\times}$,
and $T^{m-d} = T^{m+1}/T^{d+1} \cong \overline{N} \otimes
\K^{\times}$.

  The images of the basis elements $\mathbf{e}_i \in \mathbb Z^{m+1}
\cong N$ correspond to rays of $\Sigma^{\star}$.  We can thus identify
these rays with the columns of the Gale dual $D$ of the matrix $A$
(see Section~\ref{ss:equationsforquotients}).  By the construction of
$D$ the integer column space of $D$ is $\mathbb Z^{m-d} \cong \bar{N}$.

\subsection{Equations for Chow and Hilbert quotients}

In this section we show how to give equations for the Chow or Hilbert
quotient of $X$ in the Cox ring of a toric variety $X_{\Sigma}$.
Throughout this section $X$ is a $T^d$-equivariant subscheme of
$\mathbb P^m$ with no irreducible component lying in any coordinate
subspace, and we choose $X_{\Sigma}$ to be a sufficiently large toric
subvariety of the of the normalization $X_{\Sigma^\star}$ of $\mathbb
P^m \git^\star \mathbb T^d$ with respect to the pullback of $X
\git^{\star} T^d$ to $X_{\Sigma^{{\star}}}$.  We denote by $I(X)
\subseteq \K[x_0,\dots,x_m]$ the saturated ideal defining $X$.  Let
the $(d+1) \times (m+1)$ matrix $A$ record the weights of the $T^d$
action, and let $D$ be its Gale dual (see (\ref{e:Amatrix1})).  Let
$R$ be the matrix whose columns are the first integer lattice points
on the rays of $\Sigma$.  Since the columns of $D$ span the lattice
$\overline{N}$ of $\Sigma^{\star}$, one has that
$$R= D V,$$
for some $(m+1)\times r$ dimensional matrix $V$.
We denote by $X \git_n^{\star} T^d$ the pullback of $X \git^{\star}
T^d$ to the normalization $X_{\Sigma^{\star}}$ of $\mathbb P^m
\git^{\star} T^d$.

\begin{theorem} \label{t:Eqmainthm} 
Let $T^d$ act on $\mathbb P^m$ and let $X \subset \mathbb P^m$ be a
  $T^d$-equivariant subscheme with $I(X) = \langle f_1,\dots,f_g
  \rangle \subseteq \K[x_0,\dots,x_m]$.
Let $X_{\Sigma}$ be a toric subvariety of $X_{\Sigma^{\star}}$
 containing $X \git^{\star}_n T^d$, and let $S= \K[y_1, \ldots y_{r}]$
 be the Cox ring of $X_{\Sigma}$.  Let $\nu
 \colon \K[x_0,\dots,x_m] \rightarrow \K[y_1^{\pm 1}, \dots,
 y_{r}^{\pm 1}]$ be given by $\nu(x_i)=\prod_{j=1}^{r} y_j^{V_{ij}}$.

 The ideal of $X \git^{\star}_n T^d$ in $S$ is $I=\nu(I(X)) \cap S$,
 and is obtained by clearing denominators in $\{\nu(f_i): 1 \leq i
 \leq g \}$, and then saturating the result by the product of all the
 variables in $S$.
If $X \git^{\star} T^d$ is irreducible and normal, it is isomorphic to $X
\git^{\star}_n T^d$ and $I$ is the ideal of $X \git^{\star} T^d$ in
$S$.
\end{theorem} 

\begin{remark} \label{r:homogenization}
When $X_{\Sigma}$ contains all rays
of $\Sigma^{\star}$ corresponding to columns of $D$, and $D$ has no
repeated columns, we can write $V=(I | C^T)$, where $I$ is the
$(m+1)\times (m+1)$ identity matrix and $C$ is an integer $(r-m-1)
\times (m+1)$ matrix.
Thus the map $\nu$ may be thought of as homogenizing the ideal $I(X)$
with respect to the grading of $S$.  This is the formulation assumed
in Part 1 of Theorem~\ref{t:mainthm}.
\end{remark}

\begin{proof}
[Proof of Theorem~\ref{t:Eqmainthm}] The proof of the first claim
  proceeds in three parts.  First, using Proposition
  \ref{p:quotienteqtns}, we find equations in $\K[z_1^{\pm
  1},\dots,z_{m-d}^{\pm 1}]$, the coordinate ring of the torus
  $T^m/T^d$ of $X_{\Sigma}$, for $(X \cap T^m)/T^d$. 
  Second, we use the isomorphism $\rho: \K[z_1^{\pm
  1},\dots,z_{m-d}^{\pm 1}] \longrightarrow S_{(\Pi_{i=1}^{r}y_i)},$
  and apply Proposition \ref{p:closureeqtns} to obtain generators for
  the ideal in $S$ of the closure of $(X \cap T^m)/T^d$ in
  $X_{\Sigma^{\star}}$.  Since this is $X \git_n^{\star} T^d$, we last
  check that this is the same as the ideal obtained by clearing
  denominators and saturating.

The key diagram is the following, where the map $U$ is as in Proposition~\ref{p:quotienteqtns} and the $i$ are inclusions.
$$ 
\xymatrix{ 
\K[z_1^{\pm 1},\dots, z_{m+1}^{\pm 1}] \ar[r]^-{\sim}_-{U^T}  &  \K[x_0^{\pm 1},\dots, x_m^{\pm 1}] \cong \K[M] \\
\K[z_1^{\pm 1}, \dots, z_{m-d}^{\pm 1}] \ar[u]^{i}  \ar[r]^{\sim}_{D^T} \ar[d]_{\sim}^{R^T}&
\K[x_0^{\pm 1},\dots, x_m^{\pm 1}]_0 \ar[u]^i \ar[dl]^{V^T}_{\sim} \\
(S_y)_0\\
}
$$

The content of Proposition~\ref{p:quotienteqtns} is that the ideal in
$\K[z_1^{\pm 1},\dots, z_{m-d}^{\pm 1}]$ of $(X \cap T^m)/T^d$ is
given by $i^{-1}({U^T}^{-1}(I(X))$.  This is taken under the map $D^T$ to
the degree zero part of the ideal $I(X)$ in $\K[x_0^{\pm
1},\dots,x_m^{\pm 1}]$, since $U$ is chosen in
Proposition~\ref{p:quotienteqtns} to have first $d+1$ rows equal to
$D$.  We thus see that the choice of identification of the lattice
$\overline{N}$ of the fan $\Sigma$ with $\mathbb Z^{m-d}$ identifies
$\K[\overline{M}]$ with $\K[z_1^{\pm 1},\dots,z_{m-d}^{\pm 1}]$, where
$\overline{M}$ is the lattice dual to $\overline{N}$.  This means that the
isomorphism $\K[\overline{M}] \cong (S_y)_0$ given in \cite[Lemma 2.2]{Cox}
is given by the matrix $R^T$, so the function $\rho$ of
Proposition~\ref{p:closureeqtns} is given by $\rho(z_i)=\prod_{j=1}^r
y_j^{R_{ij}}$.  It thus follows from Proposition~\ref{p:closureeqtns}
that the ideal in $S$ of the closure of $(X \cap T^m)/T^d$ in
$X_{\Sigma}$ is given by applying $R^T \circ i^{-1} \circ {U^T}^{-1}$
to the generators of $I(X)$, clearing denominators, and then
saturating by the product of the variables of $S$.

To complete the proof, it thus suffices to observe that $R^T \circ
i^{-1} \circ {U^T}^{-1}$ restricted to degree zero part of $\K[x_0^{\pm
1},\dots, x_m^{\pm 1}]$ is given by the matrix $V^T$.  This follows from
the fact that $R^T=V^TD^T$ and the bottom triangle of the above
commutative diagram is made of three isomorphisms.

We now consider the case where $X \git^{\star} T^d$ is irreducible and
normal, and show that it isomorphic to $X \git^{\star}_n T^d$.  Set
$Y=X \git^{\star} T^d$, and $Z = \mathbb P^m \git^{\star} T^d$.  Let
$\widetilde{Y}$ and $\widetilde{Z} = X_{\Sigma^{\star}}$ denote the
respective normalizations.  Begin by reducing to the case where
$Z=\Spec(A)$ and $Y=\Spec(A/I)$ are irreducible affine schemes with
$I$ an ideal in a ring $A$.  Note that $Y$ intersects the torus
$T^{m-d}$ of the not-necessarily-normal toric variety $Z$, and thus
intersects the smooth locus of $Z$.  Since $Y$ is irreducible, this
means that the map $Y \times_Z \widetilde{Z} \rightarrow Y$ is
birational, and so $A$ and $A/I$ have the same total quotient ring.
Normalization is a finite morphism, and the pullback of a finite
morphism is finite, so $\widetilde{A}/I\widetilde{A}$ is integral over
$A/I$.  The isomorphism $\widetilde{A}/I\widetilde{A} \cong A/I$ then
follows from the fact that $Y$ is normal, so $A/I$ is integrally
closed.  Finally since all maps are inclusions, everything glues to
prove $Y \times_Z \widetilde{Z} \cong Y$. \end{proof}

\begin{remark}
We remark that the choice of the matrix $V$ in
Theorem~\ref{t:Eqmainthm} is not unique.  This does not affect the
computation, however, as the induced map $\nu$ is unique when
restricted to the degree zero part of $\K[x_0^{\pm 1},\dots, x_m^{\pm
1}]$.
\end{remark}

\begin{example} \label{e:finalquadric}
Theorem~\ref{t:Eqmainthm} lets us compute equations for
$\overline{M}_{0,5}$.  Kapranov's description
of $\overline{M}_{0,5}$ as the Chow or Hilbert quotient of the
Grassmannian $G(2,5)$ by the $T^4$ action, described in
Example~\ref{e:M0nasChow}, gives embeddings of $\overline{M}_{0,5}$
into the normalizations of $\mathbb P^9 \git^{Ch} T^4$ and $\mathbb
P^9 \git^{H} T^4$.  These are both five-dimensional toric varieties
whose rays include the columns of the matrix $D$ of
Example~\ref{e:torusquotient}, plus ten additional rays, being the
columns of the matrix $-D$.  Let $\Delta$ be the two-dimensional fan
with rays the columns $\mathbf{d}_{ij}$ of $D$, and cones
$\pos(\mathbf{d}_{ij}, \mathbf{d}_{kl} : \{i,j\} \cap \{k,l \} =
\emptyset)$.  By direct computation, or by
Theorem~\ref{t:definingThm}, the fan $\Delta$ is a subfan of the fan
of both $\Sigma^{Ch}$ and $\Sigma^H$ that defines a sufficiently large
toric subvariety.  For this fan we have $R=D$, so $V$ is the ${n
\choose 2} \times {n \choose 2}$ identity matrix.

The ideal in $\K[x_{ij} : 1 \leq i < j \leq 5]$ defining $G(2,5)$ as a
subscheme of $\mathbb P^9$ is generated by the Pl\"ucker equations
given in Example~\ref{e:torusquotient}.  The Cox ring of $\mathbb
P_{\Delta}$ is $\K[y_{ij} : 1 \leq i < j \leq 5]$, and the map $\nu$
of Theorem~\ref{t:Eqmainthm} is the identity map $\nu(x_{ij})=y_{ij}$.
The ideal of the normal irreducible variety $\overline{M}_{0,5}$ in
the Cox ring of $\mathbb P_{\Delta}$ is thus generated by the
Pl\"ucker relations.  The saturation step is unnecessary in this case
as this ideal is prime.

\end{example}

\section{GIT constructions of Chow/Hilbert quotients} \label{s:GIT}

In this section we give a GIT construction of the Chow and Hilbert
quotients of $X$.  This follows from their description in the Cox ring
of a toric variety in Section~\ref{s:Equations}.  We also recover all
GIT quotients of $X$ by $T^d$ by variation of GIT quotient. As before
$X$ is a $T^d$-equivariant subscheme of $\mathbb P^m$ with no
irreducible component lying in any coordinate hyperplane.

Let $X_{\Sigma}$ be a sufficiently large toric subvariety of
$X_{\Sigma^{\star}}$.  We assume in addition that $X_{\Sigma}$ has a
torsion-free divisor class group, which can be guaranteed, for
example, by taking $X_{\Sigma}$ to contain the rays of
$\Sigma^{\star}$ corresponding to columns of the matrix $D$ (see
Section~\ref{ss:equationsforquotients}).  Let $I$ be the ideal of $X
\git_n^{\star} T^d$ in the Cox ring of $X_{\Sigma}$.  We next define a
cone which will index our choices of GIT quotient.

\begin{definition} \label{d:Gcone}
  Let $\{ [D_i] : 1 \leq i \leq |\Sigma(1)| \}$ be the set of classes of the
  torus invariant divisors on $X_{\Sigma}$.  Set
$$\mathcal G(X_{\Sigma})= \bigcap_{\sigma \in \Sigma} \pos([D_i] : i \not \in \sigma).$$
\end{definition}

We note that $\mathcal{G}(X_{\Sigma})$ is the cone in $\DivCl(X_{\Sigma})\otimes \mathbb R$ of divisor classes $[D] \in
\DivCl(X_{\Sigma})$ for which $[D]$ is globally generated.

\begin{lemma}
The cone $\mathcal{G}(X_{\Sigma})$ is positive dimensional.
\end{lemma}

\begin{proof}
 Since $\Sigma$ is a subfan of the fan $\Sigma^{\star}$ we can number
  the rays of $\Sigma^{\star}$ so that the first $p$ live in $\Sigma$,
  while the last $s$ do not.  If $i : X_{\Sigma} \rightarrow
  X_{\Sigma^{\star}}$ is the inclusion map, then the pullback from
  $\DivCl(X_{\Sigma^{\star}})$ to $\DivCl(X_{\Sigma})$ is given by
  $i^{\star}([\sum_{i=1}^{p+s} a_i D_i]) = [\sum_{i=1}^{p} a_i D_i
  ]$.  Since $X_{\Sigma^{\star}}$ is a projective toric variety, its
  nef cone $\mathcal{G}(X_{\Sigma^{\star}}) = \cap_{\sigma \in
  \Sigma^{\star}} \pos( [D_i]: i \not \in \sigma)$ is a positive
  dimensional cone that can be viewed as living in
  $\DivCl(X_{\Sigma^{\star}}) \otimes \mathbb R$.  The image of
  $\mathcal{G}(X_{\Sigma^{\star}})$ under the map $i^*$ is contained
  in $\mathcal{G}(X_{\Sigma})$.  We now show that some nonzero vector
  in $\mathcal{G}(X_{\Sigma^{\star}})$ is taken to a nonzero vector
  under $i^*$.  Indeed, otherwise every element of
  $\mathcal{G}(X_{\Sigma^{\star}})$ could be written in the form
  $\sum_{i=p+1}^{p+s} a_i [D_i]$, so $\mathcal{G}(X_{\Sigma^{\star}})
  \subseteq \pos([D_i] : p+1 \leq i \leq s)$.  But by the relation
  between the cones of $\mathcal G(X_{\Sigma^{\star}})$ and chambers
  of the chamber complex (see \cite{BFS}), this means that the cone
  spanned by all the rays of $\Sigma$ lies in $\Sigma^{\star}$.
  However $X_{\Sigma}$ was assumed to be sufficiently large, which means
  that $\Sigma$  contains all cones of $\Sigma^{\star}$ intersecting the
  tropical variety of $X$, so the balancing condition on tropical
  varieties (see \cite[Theorem 2.5.1]{Speyerthesis}) implies that the
  rays of $\Sigma$ positively span the entire space $\mathbb R^{m-d}$.
  This would mean that the cone spanned by the rays of $\Sigma$ was
  all of $\mathbb R^{m-d}$, contradicting it lying in the fan
  $\Sigma^{\star}$.  We thus conclude that there are elements $v \in
  \mathcal{G}(X_{\Sigma^{\star}})$ with $i^*(v) \neq 0$, so
  $\mathcal{G}(X_{\Sigma})$ is a positive-dimensional cone.
\end{proof}

Let $l=|\Sigma(1)| - \dim(X_{\Sigma})$ be the rank of
$\DivCl(X_{\Sigma})$.  Let $H$ be the algebraic torus $\Hom(
\DivCl(X_{\Sigma}), \K^{\times})$.  Our condition that
$\DivCl(X_{\Sigma})$ is torsion-free guarantees that $H \cong
(\K^{\times})^{l}$.  We can regard $\DivCl(X_{\Sigma}) \otimes \mathbb
R$ as the space of (real) characters of the torus $H$, and
$\mathcal{G}(X_{\Sigma})$ as a subcone of the character space.  The
torus $H$ acts on $\mathbb A^r$ by $h \cdot x_i = h([D_i]) x_i$.
Recall that the torus of $X_{\Sigma}$ is $T^m / T^d$.  We denote by
$\mathrm{relint}(\mathcal{G}(X_{\Sigma}))$ the relative interior of
the cone $\mathcal{G}(X_{\Sigma})$.

Let $r=|\Sigma(1)|$.  Recall that for $\alpha \in \mathbb Z^l$ the GIT
quotient $Y \git_{\alpha} H$ of an affine variety $Y \subset \mathbb
A^{r}$ is
$$Y \git_{\alpha} H = \Proj ( \oplus_{j \geq 0}
(\K[x_1,\dots,x_r]/I(Y))_{j\alpha}),$$ where the $\mathbb Z^l$ grading
on the polynomial ring comes from the $H$-action on $\mathbb A^r$.
For $\alpha \in \mathbb Q^l$ we define $Y \git_{\alpha} H$ to be $Y
\git_{s\alpha} H$ for any integral multiple $s\alpha$.

\begin{proposition} \label{p:GIT}
Let $Y \subseteq \mathbb A^{r}$ be the subscheme defined by
the ideal $I \subseteq \Cox(X_{\Sigma})$ of $X \git_n^{\star}T^d$.
For rational $\alpha \in \mathrm{relint}(\mathcal{G}(X_{\Sigma}))$ we
have
$$X \git_{n}^{\star} T^d =Y \git_{\alpha} H.$$
\end{proposition}

\begin{proof}
  It follows from the results of Cox~\cite{Cox} and the chamber
  complex description of the secondary fan that $X_{\Sigma'} = \mathbb
  A^{r} \git_{\alpha} H$ is a projective toric variety whose fan
  $\Sigma'$ has same rays as $\Sigma$ and contains $\Sigma$ as a
  subfan.  The dense torus of $X_{\Sigma'}$ is also $T^m/T^d$.  The
  quotient $Y \git_{\alpha} H$ is a subvariety of $X_{\Sigma'}$.  Let
  $(Y \git_{\alpha} H)^0 = (Y\git_{\alpha} H) \cap T^m/T^d= (Y \cap
  T^r)/H = (X \cap T^m)/T^d$, where the last equality follows from the
  fact that $T^r/H \cong T^m/T^d$.  Then $Y \git_{\alpha} H$ is the
  closure of $(Y \git_{\alpha} H)^0$ in $X_{\Sigma'}$.  By
  Corollary~\ref{c:tevelevcorollary} to show that $Y \git_{\alpha} H$
  is the closure of $(Y \git_{\alpha} H)^0$ inside $X_{\Sigma}$ it
  suffices to show that the tropical variety of $(X \cap T^m)/T^d
  \subseteq T^m/T^d$ is contained in the support of $\Sigma$.  This
  follows from the hypothesis that $X_{\Sigma}$ is sufficiently large,
  again by Corollary~\ref{c:tevelevcorollary}.  Since $X \git_n
  ^{\star} T^d$ is the closure of $(X \cap T^m)/T^d$ in $X_{\Sigma}$,
  the proposition follows.
\end{proof}

This GIT description gives projective embeddings of $X
\git^{\star}_{n} T^d$, as we now describe.  Let $S=\K[x_1,\dots,x_r]$
be the Cox ring of $X_{\Sigma}$.  For $\alpha \in
\mathcal{G}(X_{\Sigma})$, write $S^{\alpha}$ for the subring
$\oplus_{j \geq 0} S_{j \alpha}$ of $S$.  Note that $S^{\alpha}$ has a
standard $\mathbb Z$-grading, by setting $\deg(f)=j$ for all $f \in
S_{j \alpha}$.  If $J$ is an ideal in $S$, write $J^{\alpha}$ for the
ideal $J \cap S^{\alpha}$ of $S^{\alpha}$.  The GIT description above
gives that
$$X \git^{\star}_{n} T^d = \Proj (S^{\alpha}/I^{\alpha}).$$

\begin{corollary} \label{c:alpha}
The GIT description of $X \git^{\star}_{n} T^d$ gives a projective
embedding of $X \git^{\star}_{n} T^d$ with the pullback of $\mathcal
O(1)$ on $\mathbb P^N$ equal to $\pi^*(\alpha)$ for large enough
integer $\alpha \in \mathcal G(X_{\Sigma})$, where $\pi$ is the
embedding of $X \git^{\star}_{n} T^d$ into $X_{\Sigma}$.
\end{corollary}

\begin{proof}
Assume $\alpha$ is large enough so that $S^{\alpha}/I^{\alpha}$ is
generated in degree one.  If not, we can replace $\alpha$ by $\ell \alpha$
for $\ell \gg 0$.  Let $x^{u_0},\dots,x^{u_N} \in S_{\alpha}$ generate
$S^{\alpha}/I^{\alpha}$, and define a map $\phi \colon
\K[z_0,\dots,z_N] \rightarrow S^{\alpha}/I^{\alpha}$ by
$\phi(z_i)=x^{u_i}$.  Let $J=\ker(\phi)$.  Then $X \git^{\star}_{n}
T^d= \Proj( \K[z_0,\dots,z_N]/J)$, and by construction the pullback of
$\mathcal O(1)$ is $\pi^*(\alpha)$.
\end{proof}

\begin{example} \label{e:M05GIT}
We continue the discussion of $\overline{M}_{0,5}$ begun in
Examples~\ref{e:M0nasChow} and~\ref{e:finalquadric}.
Example~\ref{e:finalquadric} shows that $\overline{M}_{0,5}$ is the
GIT quotient of the affine cone over the Grassmannian $G(2,5)$ by a
five-dimensional torus, so $\overline{M}_{0,5} = G(2,5)\git_{\alpha}
T^4$ for all $\alpha$ in the relative interior of $\mathcal
G(X_{\Delta})$.  The cone $\mathcal G(X_{\Delta})$ is the intersection
of the cones $\pos(\mathbf{d}_{pq} : \{p,q\} \neq \{i,j\}, \{k, l \}
)$ for all choices of $i,j,k,l$ distinct, where $\mathbf{d}_{pq}$ is
the column of the matrix $D$ of Example~\ref{e:torusquotient} indexed
by $\{p,q\}$.  This cone thus contains the vector $\alpha
=(2,2,2,2,2)$ in its relative interior.  The Cox ring of $X_{\Delta}$,
$S=\K[x_{ij} : 1 \leq i <j \leq 5]$, has an action of the symmetric
group $S_5$ permuting the variables, and $S^{\alpha}$ is generated in
degree one by monomials $x_{12}^2x_{34}x_{35}x_{45}$,
$x_{12}x_{23}x_{34}x_{45}x_{15}$ and their $S_5$-orbits, which are of
size $10$ and $12$ respectively.  Thus this gives an embedding of
$\overline{M}_{0,5}$ into $\mathbb P^{21}$.  The ideal $I^{\alpha}$ is
then cut out by linear equations induced from multiples of the the
Pl\"ucker relations such as $x_{15}x_{12}x_{34}x_{23}x_{45}-
x_{15}x_{12}x_{34}x_{24}x_{35}+ x_{15}x_{12}x_{34}^2x_{25}$ and its
$S_5$-orbit.  Thus $\overline{M}_{0,5}$ is cut out as a subscheme of
$\mathbb P^{21}$ by these linear relations and the binomial relations
coming from the kernel of the surjective map $\K[z_0,\dots,z_{21}]
\rightarrow S^{\alpha}$.
\end{example}

We now show that other GIT quotients of $X$ can be obtained from $X
\git^{\star}_{n} T^d$ by variation of GIT quotient (VGIT).  We assume
that $\Sigma$ contains all the rays of $\Sigma^{\star}$ corresponding
to rays of the Gale dual matrix $D$ (see
Section~\ref{ss:equationsforquotients}), and also that $D$ has no
repeated columns.  This means that we can write $R = (D | DC^T) =
D(I|C^T)$ for some $s \times (m+1)$ matrix $C$.
Let $S=\Cox(X_{\Sigma}) = \K[x_0,\dots,x_m,y_{1},\dots,y_{s}]$, where
$s=r-(m+1) = l - (d+1)$ and $\widetilde{S}=\K[x_0,\dots,x_m]$.  Grade
$\widetilde{S}$ by $\deg(x_i)=\mathbf{a}_i$, where $\mathbf{a}_i$ is
the $i$th column of the matrix $A$.  Our assumption on $\Sigma$ means
that the grading matrix for $S$ can be written in block form as
\begin{equation} \label{eqtn:GradingGIT} \left( \begin{array}{r|r}
A & 0 \\ \hline
-C & I_s \\
\end{array}
\right).
\end{equation}
We write $\mathbb N A$ for the subsemigroup of
$\mathbb Z^{d+1}$ generated by the columns of the matrix $A$.  Let
$\pi: \mathbb Z^l \rightarrow \mathbb Z^{d+1}$ be the projection onto
the first $d+1$ coordinates, and let $\pi_2 : \mathbb Z^l \rightarrow
\mathbb Z^{s}$ be the projection onto the last $s=l-(d+1)$
coordinates.  Recall the homomorphism $\nu \colon \K[x_0,\dots,x_m] \rightarrow
 \K[x_0, \dots,x_m, y_1^{\pm 1}, \dots, y_{s}^{\pm 1}]$ given in the statement
of Theorem~\ref{t:Eqmainthm}.  Since $D$ does not have repeated
columns this is given by $\nu(x^u) =x^uy^{Cu}$.  Again $Y=Z(I)$ is the
subscheme of $\mathbb A^{m+1+s}$ defined by the ideal $I$ of $X
\git^{\star}_{n} T^d$.

\begin{theorem} \label{t:VGIT}
With the notation given above, fix $\beta \in \mathbb N A$.  Then if
$\alpha \in \mathbb Z^l$ satisfies $\pi(\alpha)=\beta$, and $\alpha_i
\geq - \min \{(Cu)_i : Au=\beta, u \in \mathbb Q_{\geq 0}^{m+1} \}$ for 
$1 \leq i \leq s$,
then
$$Y \git_{\alpha} H \cong X \git_{\pi(\alpha)} T^d.$$ 
\end{theorem}

\begin{proof}
  We will show that the map $\nu$ described above induces an
  isomorphism between $\widetilde{S}^{\beta}/I(X)^{\beta}$ and
  $S^{\alpha}/I^{\alpha}$.

The map $\nu$ sends a monomial $x^u$ to $x^u y^{Cu}$.  Define $\nu'
\colon \widetilde{S}^{\beta} \rightarrow S^{\alpha}$ by setting
$\nu'(x^u) = x^u y^{\ell \pi_2(\alpha) + Cu}$, when $\deg(x^u)=\ell \beta$.
By construction if $\deg(x^u)=\ell \beta$ then $\deg(\nu'(x^u))=\ell \alpha$,
and the assumption on $\alpha$ implies that $\ell \pi_2(\alpha) +Cu \in
\mathbb N^{l-d-1}$, so the map is well-defined.  Note also that $\nu'$
is injective and surjective, as if $x^uy^v \in S_{\ell  \alpha}$, then
$\deg(x^u)=\ell \beta$, and $v$ must equal $Cu$.  We denote by
$\bar{\nu'}$ the induced map from $\widetilde{S}^{\beta}$ to
$S^{\alpha}/I^{\alpha}$, and let $J=\ker(\bar{\nu'})$.  Since
$\bar{\nu'}$ is surjective as $\nu'$ is, it remains to show that $J=
I(X)^{\beta}$.  Since $\nu'$ is a graded homomorphism, $J$ is a
homogeneous ideal, so it suffices to check that each homogeneous
polynomial in $J$ lies in $I(X)$ and vice versa.  Recall from
Theorem~\ref{t:Eqmainthm} that $I= (\nu(I(X)) S_y) \cap S$, where
$y=\prod_{i=0}^m x_i \prod_{j=1}^s y_j$.  If $f = \sum_u c_u x^u \in
(I(X))_{\ell  \beta}$, then $\nu(f) = \sum_u c_u x^u y^{Cu}$, so $\nu'(f)=
\sum c_u x^u y^{\ell \pi_2(\alpha)+Cu} = y^{\ell \pi_2(\alpha)} \nu(f) \in
I_{\ell \alpha}$, and thus $I(X)^{\beta} \subseteq J$.

Suppose now that $f \in J_{l\beta}$, so $\nu'(f) \in I_{\ell  \alpha}$,
and thus $\nu'(f) = \sum_i h_i \nu(g_i)$ for $g_i \in I(X)$ and $h_i
\in S_y$, where we may assume that the $h_i$ are constant multiples
of Laurent monomials, and the $g_i$ are homogeneous.  Write $f=\sum_u
c_u x^u$, $g_i = \sum_{v} d_{v,i} x^{v}$, and
$h_i=c_ix^{u_i^+}/x^{u_i^-}y^{v_i}$, where $u_i^+, u_i^- \in \mathbb
N^{m+1}$, and $v_i \in \mathbb Z^s$.  Note that $\nu(c_ix^{u_i^+}g_i)
= c_ix^{u_i^+}y^{Cu_i^+} \nu(g_i)$, so by changing $v_i$ and $g_i$ we
may assume that $c_i=1$ and $u_i^+=0$ for all $i$.  Also, note that
the map $\nu$ lifts to a map from $\K[x_0^{\pm 1},\dots,x_m^{\pm 1}]$
to $\K[x_0^{\pm 1}, \dots,x_m^{\pm 1}, y_1^{\pm 1},\dots, y_s^{\pm
1}]$, and $I'(X):=I(X) \K[x_0^{\pm 1}, \dots,x_m^{\pm 1}] \cap
\widetilde{S} = I(X)$ (as no irreducible component of $X$ is contained in a
coordinate subspace).  Also $\nu(I'(X))S_y=\nu(I(X))S_y$, so we may
also assume that $u_i^-=0$ for all $i$.  Thus we have $\nu'(f)=\sum_i
y^{v_i} \nu(g_i)$ where $v_i \in \mathbb Z^s$, and $g_i \in
I'(X)_{\ell \beta}$ with $\tilde{g}=\sum_i g_i \in \widetilde{S}$ because
$\tilde{g}=\nu'(f)|_{y_i=1}$ and thus $\tilde{g} \in (I(X))_{\beta}$.
Since $\deg(\nu'(f)) = \ell  \alpha$, and each $g_i$ is homogeneous, we
have $\deg(y^{v_i}\nu(g_i))=\ell \alpha$, so
$\nu'(g_i)=y^{v_i}\nu(g_i)$, and thus $\nu'(\tilde{g})=\nu'(f)$.
From this it follows that $\tilde{g}=f$, since $\nu'$ is injective by
construction, and thus $f \in I(X)_{\ell \beta}$, so $J = I(X)^{\beta}$ as
required.
\end{proof}

\begin{corollary} \label{c:VGIT}
The Chow and Hilbert quotients $X \git^{Ch} T^d$ and $X \git^{H} T^d$
and the GIT quotients $X \git_{\alpha} T^d$ are related by variation
of GIT quotient.
\end{corollary}

\begin{proof}
It remains to check that an $\alpha$ satisfying the hypotheses of
Theorem~\ref{t:VGIT} actually exists.  This means checking that the
the set $\{ (Cu)_i : Au = \beta, u \in \mathbb Q_{\geq 0}^{m+1} \}$ is
bounded below so the minimum exists.  This follows from the fact that
$\{ u: Au = \beta, u \in \mathbb Q_{\geq 0}^{m+1} \}$ is a polytope,
since we required the first row of $A$ to consist of all ones.  Thus
the lower bound is the minimum of a linear functional on a polytope,
which is finite.
\end{proof}

\section{A toric variety containing $\overline{M}_{0,n}$}
\label{s:MOn}

In the remainder of the paper we apply the previous theorems to obtain
equations for $\MOn$.  Using the construction given in
Example~\ref{e:M0nasChow}, $\MOn$ is a subvariety of the normalization
of the Chow or Hilbert quotient of $\mathbb P^{{n \choose 2} -1}$ by
$T^{n-1}$.  In this section we describe a smooth normal toric variety
$X_{\Delta}$ that is a sufficiently large toric subvariety of both
$X_{\Sigma^{Ch}}$ and $X_{\Sigma^H}$.  We prove in Theorem
\ref{t:definingThm}, that $\overline{M}_{0,n}$ is contained in
$X_{\Delta}$.

We first define a simplicial complex and set of vectors that form the
underlying combinatorics of $\Delta$.  Recall that $[n]=\{1,\dots,
n\}$, and for $I \subset [n]$ we have $I^c = [n] \setminus I$.

\begin{definition} \label{d:R}
Let $\mathcal{I}=\{I \subset [n]: 1 \in I \mbox{ and } |I|,|I^c| \geq
2\}$.  The simplicial complex  $\widetilde{\Delta}$ has vertices $\mathcal
I$ and $\sigma \subseteq \mathcal I$ is a  simplex of
$\widetilde{\Delta}$ if for all $I,J \in \sigma$ we have $I \subseteq
J$, $J \subseteq I$, or $I \cup J = [n]$.

Let $\mathcal{E}=\{ij: 2 \le i < j \le n, ij \ne 23\}$ be an indexing
set for a basis of $\mathbb{R}^{{n\choose2}-n}$ and for each $I \in
\mathcal{I}$, define the vector
$$\mathbf{r}_I=(R_{ij,I})_{ij \in \mathcal{E}} \in \mathbb{Z}^{{n\choose2}-n}, \ \ R_{ij,I} = \left\{
\begin{array}{rl}
1   &  |I \cap \{ij\}|=0,  |I \cap \{2,3\}| > 0 \\
-1  &  |I \cap \{ij\}|>0,  |I \cap \{2,3\}| =0\\
0   & \text{otherwise}\\
\end{array}
\right\}.$$ Form the $({n \choose 2}-n) \times |\mathcal I|$ matrix $R$
with $I$th column $\mathbf{r}_I$.
\end{definition}

Note that the set $\mathcal I$ also labels the boundary divisors
$\delta_I$ of $\MOn$, and a simplex $\sigma$ lies in
$\widetilde{\Delta}$ precisely if the corresponding boundary divisors
intersect nontrivially.

\begin{proposition}\label{p:Delta}
The collection of cones $\{\pos(\mathbf{r}_I: I \in \sigma): \sigma \in
\Delta\}$ is an $(n-3)$-dimensional polyhedral fan $\Delta$ in
$\mathbb{R}^{{n\choose2}-n}$.  The associated toric variety
$X_{\Delta}$ is smooth, and the support of $\Delta$ is the tropical
variety of $M_{0,n} \subset T^{{n\choose2}-n}$.
\end{proposition}

This fan is well known in the literature as the {\em space of
 phylogenetic trees} (\cite[Section 4]{SpeyerSturmfels},
 \cite{Buneman}, \cite{Vogtmann}, \cite{BilleraHolmesVogtmann}).  To
 prove Proposition~\ref{p:Delta} and for results in the remainder of
 this section, we use the following notation.  

\begin{definition} \label{d:D}
Let $A_n$ be the $n \times {n\choose2}$ dimensional matrix with
$ij$-th column equal to $e_i + e_j$ and let $D$ be the $({n \choose 2}-n) \times {n \choose 2}$ matrix 
$$D=(D_{ij,kl})_{\stackrel{ij \in \mathcal{E}}{ 1 \le k < l \le n}},
\ \ D_{ij,kl} = \left\{
\begin{array}{rl}
1   &  kl=ij,     \mbox{or }  kl  \in \{12,13\},  |\{ij\} \cap \{kl\}|=\emptyset \\
-1  &  kl=23,  \mbox{or }  kl \in \{1i,1j\},  |\{23\} \cap \{kl\}|=\emptyset\\
0   & \text{otherwise}\\
\end{array}
\right\}.$$ Let $C$ be the $(|\mathcal I|-{n \choose 2}) \times {n
\choose 2}$ matrix with rows indexed by $I \in \mathcal I$ with $3
\leq |I| \leq n-3$ and entries
$$C_{I,ij} = \left\{ \begin{array}{ll}
    1 &  i,j \in I \\
    0 & \text{otherwise}
\end{array}
\right. .$$
A straightforward calculation from the definition of $D$ shows that 
\begin{equation} \label{eqtn:RM0n}
R=D(I|C^T).
\end{equation}

\end{definition}
The matrix $A_n$ is the vertex-edge incidence matrix for the complete
graph on $n$ vertices and $D$ is its Gale dual, so $DA_n^T=0$.  Note
that the square $({n \choose 2}-n) \times ({n \choose 2}-n)$ submatrix
of $D$ with columns indexed by $\mathcal E$ is the identity matrix.
In particular, $D$ has rank ${n\choose2}-n$.

\begin{proof}[Proof of Proposition~\ref{p:Delta}]
Let the lattice $L$ be the integer row space of
$A_n$.  The affine cone over the Grassmannian $G(2,n)$ in its
Pl\"ucker embedding is a subvariety of $\mathbb A^{{n \choose 2}}$ and
we denote by $AG^0(2,n)$ its intersection with $T^{n \choose 2}$.  In
\cite{SpeyerSturmfels}, Speyer and Sturmfels show that the tropical
variety of $AG^0(2,n)$ is a $(2n-3)$-dimensional fan in $\mathbb R^{{n
\choose 2}}$ with lineality space $L\otimes \mathbb R$, and the simplicial complex
corresponding to the fan structure on $\trop(AG^0(2,n))$ is
$\widetilde{\Delta}$.  Specifically, the cone corresponding to the
cone $\sigma \in \widetilde{\Delta}$ is $\pos(\mathbf{e}_{I} : I \in
\sigma)+L$, where
$$\mathbf{e}_{I} = \sum_{i \in I, j \not \in I} \mathbf{e}_{ij} \in
\mathbb Z^{n \choose 2}.$$

Recall from Example~\ref{e:M0nasChow} that $M_{0,n} =
G^0(2,n)/T^{n-1}$.  Using Proposition~\ref{p:quotienteqtns}, one can
show that for any $X \subseteq T^m$ that is invariant under the
action of a torus $T^d \subset T^m$, the tropical variety of $X/T^d
\subset T^m/T^d$ is equal to the quotient of the tropical variety of
$X \subset T^m$ by the tropical variety of $T^d$.  Thus
$\trop(M_{0,n}) = \trop(G^0(2,n))/\trop(T^{n-1}) =
\trop(AG^0(2,n))/L$.  To prove the first part of the proposition it
then suffices to show that the image of $\mathbf{e}_I$ in $\mathbb
R^{{n \choose 2}}/L$ is $\mathbf{r}_I$.  
 Since the matrix $D$ of Definition~\ref{d:D} is a Gale dual for
$A_n$, the map $\mathbb R^{n \choose 2} \rightarrow \mathbb R^{n
\choose 2}/L \cong \mathbb R^{{n \choose 2}-n}$ given by sending the
basis vector $\mathbf{e}_{ij}$ of $\mathbb R^{n \choose 2}$ to $(-1/2)
\mathbf{d}_{ij}$ is an isomorphism, where $\mathbf{d}_{ij}$ is the
$ij$th column of $D$.  The image of $\mathbf{e}_I$ under this map is
then $-1/2 \sum_{i \in I, j \not \in I} \mathbf{d}_{ij}$.  Since
$\sum_{1 \leq i <j \leq n} \mathbf{d}_{ij} = \mathbf{0}$, this is $1/2
\sum_{i, j \in I} \mathbf{d}_{ij} + 1/2 \sum_{i,j \in I^c}
\mathbf{d}_{ij}$.  Let $\ell_i$ denote the $i$th row of the matrix
$A_n$, and note that $\sum_{i \in I} \ell_i - \sum_{i \not \in I}
\ell_i = \sum_{i,j \in I} \mathbf{e}_{ij} - \sum_{i, j \in I^c}
\mathbf{e}_{ij}$, so $\sum_{i,j \in I} \mathbf{d}_{ij} = \sum_{i,j \in
I^c} \mathbf{d}_{ij}$.  Thus the image of $\mathbf{e}_I = \sum_{i,j
\in I} \mathbf{d}_{ij} = \sum_{i,j \in I^c} \mathbf{d}_{ij}$.  Since
this is equal to $(D(I|C^T))_{I} = \mathbf{r}_I$, we conclude that
$\mathbf{r}_I$ is the image of $\mathbf{e}_I$.  Note that this map is
induced from a map of lattices, as the relevant lattice in $\mathbb
R^{n \choose 2}$ is the index two sublattice of $\mathbb Z^{n \choose
2}$ with even coordinate sum.

The fact that $\Delta$ is simplicial of dimension $n-3$ is due to
\cite{RobinsonWhitehouse}.  Recall that a fan is smooth if for each
cone the intersection of the lattice with the linear span of that cone
is generated by the first lattice points on each ray of the cone.  The
fact that the fan is smooth follows from the work of Feichtner and
Yuzvinsky.  In \cite{Feichtner} it is shown that the fan $\Delta$ is
the one associated to the nested set complex for a related hyperplane
arrangement, while in \cite[Proposition 2]{FeichtnerYuzvinsky} it is
shown that the the fans associated to nested set complexes are smooth.  
\end{proof}

\begin{notation} Let $A_n$ and $D$ be the matrices described in Definition \ref{d:D}. 
 Recall that the {\em chamber complex} $\Sigma(D)$ is the polyhedral
fan in $\mathbb R^{{n\choose2}-n}$ subdividing the cone spanned by the
columns of $D$ obtained by intersecting all simplicial cones spanned
by columns of $D$.  This is equal to the secondary fan of $A_n$, and
thus to the fan of the toric variety $X_{\Sigma^{Ch}}$ (see
\cite{BFS}).  Recall also that the regular subdivision $\Delta_w$ of
the configuration of the columns $\mathbf{a}_{ij}$ of $A_n$
corresponding to a vector $w \in \mathbb R^{{n \choose 2}}$ has
$\pos(\mathbf{a}_{ij} : ij \in \sigma)$ as a cell for $\sigma
\subseteq \{ ij : 1 \leq i < j \leq n \}$ if and only if there is some
$\mathbf{c} \in \mathbb R^n$ such that $\mathbf{c} \cdot
\mathbf{a}_{ij} = w_{ij}$ for $ij \in \sigma$, and $\mathbf{c} \cdot
\mathbf{a}_{ij} < w_{ij}$ for $ij \not \in \sigma$ (see~\cite{GKZ} and
\cite{GBCP} or \cite{IndiaNotes} for background on regular
subdivisions).  We denote by $\mathbb N A_n$ the subsemigroup of
$\mathbb N^n$ generated by the columns of $A_n$, and by $\mathbb
R_{\geq 0} A_n$ the cone in $\mathbb R^n$ whose rays are the positive
spans of the columns of $A_n$.
\end{notation}

\begin{proposition} \label{p:deltasubfan} The fan ${\Delta}$ is a subfan of the 
fan $X_{\Sigma^{Ch}}$, so $X_{\Delta}$ is a toric subvariety of
$X_{\Sigma^{Ch}}$.
\end{proposition}

\begin{proof}
  Let $\sigma$ be a top-dimensional
cone of $\Delta$.  Then $\sigma= \pos(\mathbf{r}_{I_1},\dots,
\mathbf{r}_{I_{n-3}} )$, where $I_i \subseteq I_j$, $I_j \subseteq
I_i$, or $I_i \cap I_j = [n]$.  There is a trivalent (phylogenetic)
tree $\tau$ with $n$ labeled leaves such that the $I_j$ correspond to
the splits obtained by deleting internal edges of $\tau$.

Since $\sigma$ is a cone in $V=\mathbb R^{{n \choose 2}} /
\mathrm{row}(A)$, we can chose a lift $w \in \widetilde{V}=\mathbb R^{{n \choose
2}}$ for a vector in $\sigma$, and thus consider the regular
subdivision $\Delta_w$ of the configuration $A_n$, which does not
depend on the choice of lift.

To show that $\sigma$ is a cone in the chamber complex of $D$ we first
characterize the subdivision $\Delta_w$ coming from a lift $w$ of a
vector in $\sigma$.  To an internal vertex $v$ of the phylogenetic
tree $\tau$ we associate the set $\mathcal C_v$ of pairs $ij$ such
that the path in $\tau$ between the vertices labeled $i$ and $j$
passes through $v$, and the cone $C_v=\pos(\mathbf{e}_{ij} : ij \in
\mathcal C_v) \subseteq \mathbb R^n$.
The cone $C_v$ is obtained by taking the cone over the polytopes
$M(\cong_v)$ of \cite[Remark 1.3.7]{Kapranov}.

  We claim that for any lift $w \in \mathbb R^{{n \choose 2}}$ of a
vector in $\sigma$ the subdivision $\Delta_w$ has cones $C_v$ for $v$
an internal vertex of $\tau$.  To see this, for each $\mathbf{r}_I \in
\sigma$ set $I'=[n] \setminus I$ if the path from leaf $i$ to leaf $j$
in $\tau$ passes through $v$ for some $i,j \in I$, and $I' = I$
otherwise.  This ensures that $ij \not \in \mathcal C_v$ for $i,j \in
I'$.  If $\sum_{i=1}^{n-3} a_i \mathbf{r}_{I_i}$ is a point in
$\sigma$ with $a_i>0$ for all $i$, then we can choose the lift $w =
\sum_{i=1}^{n-3} a_i \sum_{i,j \in I'_i } \mathbf{e}_{ij}$.  Since
$w_{ij}>0$ for $ij \not \in \mathcal C_v$,  and $w_{ij}=0$ for $ij \in
\mathcal C_v$, taking $\mathbf{c}=\mathbf{0}$ we see that $C_v$ is a
face of $\Delta_w$.

This shows that the given collection are cones in the subdivision
$\Delta_w$.  The fact that they cover all of the cone generated by the
$\mathbf{a}_{ij}$ is Claim 1.3.9 of \cite{Kapranov}.  For the readers'
convenience we give a self-contained proof.  To show that we are
not missing anything, it suffices to show that any $\mathbf{v} \in
\mathbb N A_n$ lies in $C_v$ for some $C_v$, which will show that the
$C_v$ cover $\mathbb R_{\geq 0} A_n$.  If $\mathbf{v}$ lies in
$\mathbb N A_n$ then there is a graph $\Gamma$ on $n$ vertices with degree
sequence $\mathbf{v}$.  We may assume that $\Gamma$ has the largest edge
sum out of all graphs with degree sequence $\mathbf{v}$, where an edge
$ij$ has weight the number of internal edges in the path between $i$ and $j$ in
the tree $\tau$.  Note that this means that if $ij$ and $kl$ are two
edges of $\Gamma$ with $|\{i,j,k,l\}|=4$ then the paths in $\tau$
corresponding to these two edges must cross, as otherwise we could get
a larger weight by replacing these two edges by the pair with the same
endpoints that do cross.  We claim that there is then some vertex $v
\in \tau$ for which the path in $\tau$ corresponding to each edge in
$\Gamma$ passes through $v$, which will show that $\mathbf{v} \in C_v$.
The statement is trivial if $\Gamma$ has at most two edges, since any two
paths must intersect.  The set of collections of edges of $\Gamma$ for
which the corresponding paths share a common vertex forms a simplicial
complex, so if the claim is false, we can find a subgraph $\Gamma'$ of $\Gamma$
for which there is no vertex of $\tau$ through which all of the
corresponding paths pass, but every proper subgraph of $\Gamma'$ has the
desired property ($\Gamma'$ is a minimal nonface of the simplicial
complex).  The subgraph $\Gamma'$ must have at least three edges.  Pick
three edges $e_1, e_2, e_3$ of $\Gamma'$, and let $v_i$ for $i \in
\{1,2,3\}$ be a vertex of $\tau$ for which the path corresponding to
each vertex of $\Gamma'$ except $e_i$ passes.  Since $\tau$ is a tree one
of these vertices lies in the path between the other two; without loss
of generality we assume that $v_2$ lies between $v_1$ and $v_3$.  But
the path corresponding to $e_2$ passes through $v_1$ and $v_3$ while
avoiding $v_2$, a contradiction.

This shows that the subdivision corresponding to a vector in the
interior of $\sigma$ is the same subdivision $\Delta_{\sigma}$ for any
vector in $\sigma$, and that $\sigma$ lies inside the cone $\cap_{v}
\pos(\mathbf{r}_{ij} : ij \not \in \mathcal C_v)$ of the chamber
complex of $D$.  To finish the proof, we show that this cone lies in
$\sigma$.  The cone $\sigma$ has the following facet description: for
each quadruple $i,j,k,l$ two of the pairs of paths $\{ij,kl\}$,
$\{ik,jl\}$, and $\{il,jk\}$ in the tree $\tau$ have the same combined
length, and one pair is shorter.  Without loss of generality we assume
that $ij, kl$ is the shorter pair.  This gives the inequality
$w_{ij}+w_{kl} \geq w_{ik}+w_{jl} = w_{il}+w_{jk}$.  The set of these
inequalities as $\{i,j,k,l\}$ ranges over all $4$-tuples of $[n]$
gives a description of the lift in $\mathbb R^{n \choose 2}$ of the
cone $\sigma$ (see \cite[Theorem 4.2]{SpeyerSturmfels}).  Note that
$\mathrm{row}(A)$ lies in this cone, so to show that a vector
$\mathbf{v}$ in the relative interior of $\cap_v \pos(\mathbf{r}_{ij}
: ij \not \in \mathcal C_v)$ lies in $\sigma$, it suffices to show
that for each inequality there is some lift of $\mathbf{v}$ to
$\mathbb R^{{n \choose 2}}$ that satisfies that inequality.  Given
such a $\mathbf{v}$, and a $4$-tuple $\{i,j,k,l\}$ giving the
inequality $w_{ij}+w_{kl} \geq w_{ik}+w_{jl} = w_{il}+w_{jk}$, pick a
vertex $v$ in $\tau$ that lies in all of the paths $ik$, $jl$, $il$,
and $jk$.  Then since $\mathbf{v} \in \mathrm{relint}(\pos(
\mathbf{r}_{ij} : ij \not \in \mathcal C_v))$, there is a lift $w$ of
$\mathbf{v}$ with $w_{ik}=w_{jl}=w_{il}=w_{jk} =0$, and $w_{ij},
w_{kl} \geq 0$, which satisfies the given inequality.  We conclude
that $\cap_{v} \pos(\mathbf{r}_{ij} : ij \not \in \mathcal C_v)
\subseteq \sigma$, and thus we have equality, so $\sigma$ is a cone in
the chamber complex of $D$.
\end{proof}

Recall that an open cone $\sigma \in \Sigma^H$ consists of $w$ for
which the saturation of $\inn_w(I_A)$ is constant, where we here
follow the standard convention that the leading form $\inn_w(f)$ of a
polynomial $f$ consists of terms of largest $w$-degree.

\begin{proposition} \label{p:deltaHsubfan} The fan $\Delta$ is a subfan of the fan $\Sigma^H$, so $X_{\Delta}$ is a toric subvariety of 
 $X_{\Sigma^{H}}$.
\end{proposition}

\begin{proof}
Continuing with the notation from the proof of
Proposition~\ref{p:deltasubfan}, we first show that $\inn_w(I_{A_n})$
is constant for all lifts $w$ of vectors in a maximal cone $\sigma$ of
$\Delta$ corresponding to a phylogenetic tree $\tau$.  A planar
representation of $\tau$ with the vertices on a circle determines a
circular order on $[n]$.  Without loss of generality we may assume
that this is the standard increasing order.  Draw the complete graph
$K_n$ on the circle with the same order.  By \cite[Theorem 9.1]{GBCP}
a reduced Gr\"obner basis $\mathcal B$ for $I_{A_n}$ is given by
binomials of the form $x_{ij} x_{kl} - x_{ik} x_{jl}$, where the edges
$ik$ and $jl$ of $K_n$ cross but the edges $ij$ and $kl$ do not.  The
open Gr\"obner cone corresponding to this is $\mathcal C=\{ w \in
\mathbb R^{n \choose 2} : w_{ij} + w_{kl} > w_{ik} + w_{jl} \}$.  Note
that, as in the proof of Proposition~\ref{p:deltasubfan}, the lift of
any vector in $\sigma$ lies on the boundary of $\mathcal C$, and so
there is a term order $\prec$ for which
$\inn_{\prec}(\inn_w(I_{A_{n}})) = \langle x_{ij} x_{kl} : ij, kl
\text{ do not cross in } K_n \rangle$, where $w$ is a lift of any
vector in $\sigma$.  Thus by \cite[Corollary 1.9]{GBCP}, a Gr\"obner
basis for $I_{A_n}$ with respect to the weight order given by such a
$w$ is obtained by taking the initial terms with respect to $w$ of the
Gr\"obner basis $\mathcal B$.  For a binomial $x_{ij} x_{kl} - x_{ik}
x_{jl}$, where the edges $ik$ and $jl$ of $K_n$ cross but the edges
$ij$ and $kl$ do not, either the paths $ik$ and $jl$ also cross in
$\tau$, or they do not.  If the paths do cross, the lift $w$ of a
vector in $\sigma$ has $w_{ik}+w_{jl}=w_{ij}+w_{jl}$, and if they do
not there is an internal edge of $\tau$ in the paths $ik$ and $jl$ but
not the paths $ij$ and $kl$, so if $w$ lies in the interior of
$\sigma$ we have $w_{ik}+w_{jl} < w_{ij}+w_{jl}$.  This means that the
initial ideal is determined by the tree $\tau$, so $\inn_w(I_{A_n})$
is constant for all lifts $w$ of vectors in $\sigma$.  This shows that
$\sigma$ is contained in a cone $\sigma'$ of $\Sigma^H$.

It remains to show that $\sigma = \sigma'$.  This follows from
Proposition~\ref{p:deltasubfan}, which shows that $\sigma$ is a cone
in the secondary fan of $I_{A_n}$, since the Gr\"obner fan refines the
secondary fan \cite[Proposition 8.15]{GBCP}.  The fan $\Sigma^H$ is
obtained from the Gr\"obner fan of $I_{A_n}$ by amalgamating cones
corresponding to initial ideals with the same saturation with respect
to $\langle x_{ij} : 1 \leq i < j \leq n \rangle$, so the result
follows, since ideals with the same saturation have the same radical,
and thus the cones live in the same secondary cone.
\end{proof}

We are now able to show that $\overline{M}_{0,n}$ is a subvariety of
$X_{\Delta}$.  
As described in Example~\ref{e:M0nasChow} the moduli space $\MOn$ is
both the Chow and Hilbert quotient of the Grassmannian $G(2,n)$ by the
torus $T^{n-1}$.   

\begin{theorem} \label{t:definingThm} The toric variety $X_{\Delta}$
  is the union of those $T^{{n \choose 2}-n}$-orbits of
  $X_{\Sigma^{\star}}$ intersecting the closure of $M_{0,n}$ in this
  $X_{\Sigma^{\star}}$.  The closure of $M_{0,n} \subseteq T^{{n
  \choose 2}-n}$ inside $X_{\Delta}$ is equal to $\MOn$.
\end{theorem}

We remark that the second part of this result was originally observed
by related methods in \cite[Theorem 5.5]{Tevelev}.

\begin{proof}[Proof of Theorem~\ref{t:definingThm}]
By Propositions~\ref{p:deltasubfan} and \ref{p:deltaHsubfan} we know
that $\Delta$ is a subfan of both $\Sigma^{Ch}$ and $\Sigma^H$, and by
Proposition~\ref{p:Delta} we know that the support of $\Delta$ is the tropical
variety of $M_{0,n} \subset T^{{n \choose 2}-n}$.
Corollary~\ref{c:tevelevcorollary} says that the closure of $M_{0,n}
\subset T^{{n \choose 2}-n}$ inside $X_{\Sigma^{\star}}$ intersects
the orbit corresponding to a cone $\sigma \in \Sigma^{\star}$ if and
only the tropical variety of $M_{0,n}$ intersects the interior of
$\sigma$.  Thus the orbits of $X_{\Sigma^{\star}}$ intersecting this
closure are precisely the orbits in the toric subvariety $X_{\Delta}$.
This proves the first assertion of the theorem.  For the second, note
that the closure of $M_{0,n}$ in $X_{\Delta}$ is 
$G(2,n) \git_n^{\star} T^{n-1}$, and
since $\MOn=G(2,n) \git^{\star} T^{n-1}$ is smooth and irreducible, this 
is equal to $\MOn$ by Theorem~\ref{t:Eqmainthm}.
\end{proof}

\begin{remark} We emphasize that although we know that the fans $\Sigma^{Ch}$
 and $\Sigma^H$ are the secondary and saturated Gr\"obner fans
respectively, we do not even know how many rays each has.  So while in
theory one could describe equations for $\overline{M}_{0,n}$ in the
Cox rings of these toric varieties, this is not possible in practice.
On the other hand, $\Delta$ has a completely explicit description for
all $n$.  In particular, this makes the Cox ring of $X_{\Delta}$
accessible and enables us to derive equations for $\overline{M}_{0,n}$
inside it.
\end{remark}

\section{Equations for $\overline{M}_{0,n}$} \label{equationssection}
\label{s:MOnProof}

In this section we apply Theorem~\ref{t:Eqmainthm} to to give
equations for $\overline{M}_{0,n}$ in the Cox ring of the toric
variety $X_{\Delta}$ described in Section~\ref{s:MOn}.

Recall that $\mathcal I = \{ I \subset [n]: 1 \in I, |I|, |I^c| \geq 2
\}$, and there is a ray $\pos(\mathbf{r}_I)$ of $\Delta$ for each $I
\in \mathcal I$.  The Cox ring of $X_{\Delta}$ is $S=\K[x_I : I \in
\mathcal I]$, with $\deg(x_I)=[D_I]\in \DivCl(X_{\Delta})$, where
$D_I$ is the torus-invariant divisor corresponding to the ray through
$\mathbf{r}_I$.  We construct the grading matrix as follows.  Set $b =
|\mathcal I|-{n \choose 2}$ for $n \geq 5$.  Let $G$ be the $(n+b)
\times |\mathcal I|$ matrix that is given in block form by
\begin{equation} \label{eqtn:G} G=\left( \begin{array}{rrr}
A_n & | & 0 \\ \hline
-C & | & I_{b} \\
\end{array}
\right) ,
\end{equation} where $0$ denotes the $n \times b$ zero matrix, $A_n$ and $C$ are as in Definition~\ref{d:D}, and  $I_{b}$ is
the $b \times b$ identity matrix.  For $n=4$ we set $G=(1 \, 1 \, 1)$.

As the following lemma shows, this form of $G$ is consistent with the
choice of grading matrix for the Cox ring of $X_{\Delta}$ suggested in
Equation~\ref{eqtn:GradingGIT}.  The lemma also shows that
$\DivCl(X_{\Delta})\cong \Pic(\MOn)$, so we first recall the
description of $\Pic(\MOn)$.  For $I \subset [n]$ with $1 \not \in I$
the notation $\mathbf{e}_I$ means the basis element $\mathbf{e}_{I^c}$
of $\mathbb Z^{|\mathcal I|}$.

\begin{proposition} \cite[Theorem 1]{KeelTAMS}
Let $W$ be the sublattice of $\mathbb Z^{|\mathcal I|}$ spanned by the
vectors
$$w_{ijkl}=\sum_{i,j \in I, k,l \not \in I} \mathbf{e}_I - \sum_{i,l \in I,
j, k \not \in I} \mathbf {e}_I$$
where $\{ i, j, k, l \} \subseteq [n]$ has size four.  The Picard
group of $\MOn$ is isomorphic to $\mathbb Z^{|\mathcal I|} / W$.  
\end{proposition}

\begin{lemma} \label{l:ClequalsPic}
For $n \geq 5$ the divisor class group of $X_{\Delta}$ is isomorphic to $\mathbb
Z^{b+n}$, with the image of $[D_I]$ under this isomorphism equal to
the column $\mathbf{g}_I$ of the matrix $G$ indexed by $I \in \mathcal
I$.  We have $\DivCl(X_{\Delta}) \cong \Pic(\MOn)$, with the
isomorphism taking $[D_I]$ to the boundary divisor $\delta_I$.
\end{lemma}

\begin{proof}
  The fact that $\DivCl(X_{\Delta}) \cong \mathbb Z^{b+n}$ follows
  from the short exact sequence ($\dag$) computing the class group of
  a toric variety \cite[p63]{Fulton}, and Proposition~\ref{p:Delta},
  since smooth toric varieties have torsion-free divisor class groups.
  To see that the image of $[D_I]$ is $\mathbf{g}_I$, the $I$th column
  of $G$, it suffices to show that the matrix $G$ is a Gale dual for
  the matrix $R$, so the exact sequence ($\dag$) is
$$0 \rightarrow M \stackrel{R^T}{\rightarrow} \mathbb Z^{|\mathcal I|}
\stackrel{G}{\rightarrow} \DivCl(X_{\Delta}) \rightarrow 0.
$$
Now $GR^T= G (I | C^T)^T D^T$, so 
$$GR^T = \left( \begin{array}{rrr} A_n & | & 0 \\ \hline
-C & | & I \\
\end{array}
\right)
\left(
\begin{array}{l}
I \\ \hline
C \\
\end{array}
\right)
D^T = \left( \begin{array}{c}
AD^T \\ \hline
0 \\
\end{array}
\right) = 0.$$

Finally, to show that $\DivCl(X_{\Delta}) \cong \Pic(\MOn)$, since
$\DivCl(X_{\Delta}) \cong \mathbb Z^{|\mathcal I|}/\im_{\mathbb Z}
R^T$, and $\Pic(\MOn) \cong \mathbb Z^{|\mathcal I|}/W$, it suffices
to show that $W=\im_{\mathbb Z} R^T$.  Since $\im_{\mathbb Z} R^T =
\ker_{\mathbb Z} G$, both lattices $W$ and $\im R^T$ are saturated,
and $\rank W = \rank R^T = {n \choose 2} -n$, it suffices to show that
each $w_{ijkl}$ lies in the kernel of $G$.  Now $w_{ijkl}$ restricted
to the sets $I$ with $|I|=2$ or $|I|=n-2$ is
$\mathbf{e}_{ij}+\mathbf{e}_{kl}-\mathbf{e}_{il}-\mathbf{e}_{jk}$,
which lies in $\ker(A_n)$.  Restricting $w_{ijkl}$ to $\{
\mathbf{e}_{st}: s,t \in I \} \cup \{\mathbf{e}_I \}$ we see that
$\left(
 \begin{array}{rrr} -C & | & I_{b} \\ \end{array} \right) w_{ijkl} =0$.  
  For example, if $1,i,j,k,l \in I$
then the sum is $-1-1+1+1+0=0$. Thus $w_{ijkl} \in \ker(G)$ as required.
\end{proof}

For $n=4$ we have $\DivCl(X_{\Delta}) \cong \Pic(\overline{M}_{0,4})
\cong \mathbb Z$, and the image of the $\mathbf{g}_{ij}$ is $1 \in
\mathbb Z$, which is also equal to each $[D_{ij}]$.

 We begin by proving the first part of Theorem~\ref{t:M0nmainthm}.
Recall that $\mathcal I = \{ I \subset [n]: 1 \in I, |I|, |I^c| \geq 2
\}$.  If $1 \not \in I \subseteq [n]$ then by $x_I$ we mean the
variable $x_{[n] \setminus I}$ in the Cox ring $S=\K[x_I : I \in
\mathcal I]$ of $X_{\Delta}$.

\begin{theorem} \label{t:M0nmainthm2} 
  For $n \geq 5$ the equations for $\MOn$ in the Cox ring of
$X_{\Delta}$ are obtained by homogenizing the Pl\"ucker relations with
respect to the grading of $S$ and then saturating by the product of
the variables of $S$.  Specifically, the ideal is
$$I(\overline{M}_{0,n})=\left( \left \langle \prod_{i,j \in I, k,l \not \in
I} x_I - \prod_{i,k \in I, j,l \not \in I} x_I + \prod_{i,l \in I, j,k
\not \in I} x_I \right \rangle : (\prod_{I} x_I)^{\infty}\right),$$ where the
generating set runs over all $\{i,j,k,l\}$ with $1 \leq i < j < k <l
\leq n$.
\end{theorem}

Before proving the theorem, we first find equations for the
intersection $M_{0,n}$ of $\overline{M}_{0,n}$ with the torus $\mathbb
T=T^{{n \choose 2}-1}/T^{n-1} \cong (\K^{\times})^{{n \choose 2}-n}$
of the toric variety $X_{\Delta}$.  The coordinates for $\mathbb T$
are labeled by $\mathcal E=\{ij: 2 \le i < j \le n, \ \ ij \ne
23\}$.  Recall that the ideal of $G(2,n) \subset \mathbb P^{{n \choose 2}-1}$ is generated by the Pl\"ucker relations $p_{ijkl}=x_{ij}x_{kl}-x_{ik}x_{jl}+x_{il}x_{jk}$:
$$I_{2,n} = \langle x_{ij}x_{kl}-x_{ik}x_{jl}+x_{il}x_{jk} : 1 \leq
i<j<k<l \leq n \rangle.$$

\begin{proposition} \label{l:idealintorus}
The intersection $M_{0,n}=(Z(I_{2,n})\cap T^{{n\choose 2}})/T^n
\subseteq \mathbb T$ is cut out by the equations
$$J=\langle z_{kl}-z_{2l}+z_{2k} : 3 \leq k \leq l \leq n \rangle
\subseteq \K[z_{ij}^{\pm 1} : ij \in \mathcal E],$$ where
we set set $z_{23}=1$.
\end{proposition}

\begin{proof}
We first show that $(Z(I_{2,n}) \cap T^{n \choose 2})/T^n$ is defined
by the ideal
$$J'=\langle z_{ij}z_{kl}-z_{ik}z_{jl}+z_{il}z_{jk} : 1 \leq i < j <k<l
\leq n \rangle \subseteq \K[z_{ij}^{\pm 1} : ij \in \mathcal E],$$
where we set $z_{ij}=1$ when $ij \not \in \mathcal E$.  The content of
Proposition~\ref{p:quotienteqtns} is that the relevant ideal is
$\phi^{-1}(I_{2,n})$, where $\phi \colon \K[z_{ij}^{\pm 1} : ij \in
\mathcal E] \rightarrow \K[x_{ij}^{\pm 1} : 1 \leq i < j \leq n]$ is given by
$\phi(z_{ij})=\prod_{kl} x_{kl}^{D_{ij,kl}}$.

Since the map $\phi$ is an injection, to show that $J'$ is the desired
ideal we just need to show that the generators of $J'$ are taken to a
generating set for $I_{2,n} \subseteq \K[x_{ij}^{\pm 1}]$ by $\phi$.
When $ij \in \mathcal E$ we have 
$$\phi(z_{ij})  = \left\{
\begin{array}{ll}
(x_{ij}x_{12}x_{13})/(x_{1i}x_{1j}x_{23}) & i,j \geq 4 \\
(x_{2j}x_{13})/(x_{1j}x_{23}) & i=2\\
(x_{3j}x_{12})/(x_{1j}x_{23}) & i=3\\
\end{array}
\right. .
$$ The proof breaks down into several cases, depending on how many of
$1,2,3$ lie in $\{i,j,k,l\}$.  For example, if $1,2,3 \not \in
\{i,j,k,l\}$, then $z_{ij}z_{kl}-z_{ik}z_{jl}+z_{il}z_{jk} =
(x_{ij}x_{kl}-x_{ik}x_{jl}+x_{il}x_{jk})(x_{12}x_{13})^2/(x_{1i}x_{1j}x_{1k}x_{1l}x_{23}^2)$.
The other cases to check are:
\begin{enumerate}
\item $i=1$, $2,3 \not \in \{j,k,l\}$,
\item $i=1$, $j=2$, $3 \not \in \{k,l\}$, or $j=3$, $2 \not \in \{k,l\}$,
\item $i=1$, $j=2$, $k=3$,
\item $i=2$, $1,3 \not \in \{j,k,l\}$ or $j=3$, $1,2 \not \in \{i,k,l\}$.
\item $i=2$, $j=3$, $1 \not \in \{k,l\}$.
\end{enumerate}
In every case we see that the polynomial
$\phi(z_{ij}z_{kl}-z_{ik}z_{jl}+z_{il}z_{jk})$ is equal to
$x_{ij}x_{kl}-x_{ik}x_{jl}+x_{il}x_{jk}$ times a monomial in the
$x_{mn}$.  This shows that $\phi$ takes a generating set for $J'$ to a
generating set for $I_{2,n} \subset \K[x_{ij}^{\pm 1} : 1 \leq i < j
\leq n]$, and thus $J'$ is the ideal of $(Z(I_{2,n})\cap T^{{n\choose
2}})/T^n$

To see that $J=J'$, it suffices to show that all other generators of
$J'$ lie in the ideal generated by these linear ones.  Indeed,
\begin{align*}
z_{ij}z_{kl}-z_{ik}z_{jl}+z_{il}z_{jk} =
z_{ij} (z_{kl}-z_{2l}+z_{2k})  - z_{ik} (z_{jl}-z_{2l}+z_{2j}) 
+ z_{il} (z_{jk}-z_{2k}+z_{2j}) \\
+ (z_{2l}-z_{2k}) (z_{ij}-z_{2j}+z_{2i}) 
+ (z_{2k}-z_{2j}) (z_{il}-z_{2l}+z_{2i}) + (z_{2j}-z_{2l}) (z_{ik}-z_{2k}+z_{2i}). 
\end{align*} \end{proof}

\begin{remark}
 Consider the ideal $\widetilde{J} \subset \K[z_{ij} : ij \in \mathcal
E \cup \{23\}]$ obtained by homogenizing the ideal $J$ by adding the
variable $z_{23}$.  The variety  $Z(\widetilde{J}) \subseteq \mathbb
P^{{n-1 \choose 2}-1}$ is the linear subspace equal to the row space
of the $(n-1) \times {n -1 \choose 2}$ matrix $\widetilde{A}$ whose
columns are the positive roots of the root system $A_{n-2}$.
Specifically, the rows of $\widetilde{A}$ are indexed by $2, \dots,
n$, and the columns are indexed by $\{\{i,j\} : 2 \leq i <j \leq n
\}$, with $\widetilde{A}_{i,\{j,k\}}$ equal to $1$ if $i=j$, $-1$ if
$i=k$, and zero otherwise. The variety $M_{0,n}$ is the intersection
of $Z(\widetilde{J}) \subseteq \mathbb P^{{n-1 \choose 2}-1}$ with the
torus $T^{{n-1 \choose 2}-1}$ of $\mathbb P^{{n-1 \choose 2}-1}$.
This exhibits $M_{0,n}$ as a hyperplane complement and as a very affine
variety in its intrinsic torus~\cite{Tevelev}.
\end{remark}

\begin{proof}[Proof of Theorem \ref{t:M0nmainthm2}]
By Theorem~\ref{t:definingThm} $X_{\Delta}$ is a sufficiently large
toric subvariety of $X_{\Sigma^{\star}}$, so Theorem~\ref{t:Eqmainthm}
describes how to get equations for $\MOn$ inside $X_{\Delta}$.
Equation~\ref{eqtn:RM0n} defines $V=(I_{n \choose 2} | C^T)$, so the
map $\nu \colon \K[x_{ij} : 1 \leq i < j \leq n] \rightarrow
\K[x_I^{\pm 1} : I \in \mathcal I]$ of Theorem~\ref{t:Eqmainthm} is
given by $\nu(x_{ij}) = x_{ij} \prod_{i,j \in I} x_I$, where the
product is over $I \in \mathcal I$ with $1 \in I$, and $3 \leq |I|
\leq n-3$.  Thus for $i,j,k,l$ distinct, $\nu(x_{ij}x_{kl}) =
x_{ij}x_{kl} \prod_{ij \in I} x_I \prod_{kl \in I} x_I$, with the same
restrictions on the products.  Using the convention $x_I = x_{[n]
\setminus I}$, we can write this as $\nu(x_{ij}x_{kl})= x_{ij}x_{kl}
\prod_{ij \in I, kl \not \in I} x_I \prod_{1 \in I, |\{i,j,k,l\} \cap
I| \geq 3} x_I$, where here the products are over all $I \in \mathcal
I$ with $3 \leq |I| \leq n-3$ (no restriction that $1 \in I$).  So
$\nu(p_{ijkl}) = (\prod_{ij \in I, kl \not \in I} x_I - \prod_{ik \in
I, jl \not \in I} x_I + \prod_{il \in I, jk \not \in I} x_I) \prod_{1
\in I, |\{i,j,k,l\} \cap I| \geq 3} x_I$.  Note that $\nu(p_{ijkl})$
is already a polynomial, so there is no need to clear denominators.
Thus by Theorem~\ref{t:Eqmainthm} the ideal $I_{\MOn}$ in the Cox ring
of $X_{\Delta}$ is given by:
\begin{align*} I_{\MOn} & = & \left(\left \langle (\prod_{ij \in I, kl \not \in I} x_I - \prod_{ik
\in I, jl \not \in I} x_I + \prod_{il \in I, jk \not \in I} x_I)
(\prod_{1 \in I |\{i,j,k,l\} \cap I| \geq 3} x_I) \right \rangle : (\prod_{I \in \mathcal I} x_I)^{\infty} \right)\\ 
& = & \left(\left\langle
\prod_{ij \in I, kl \not \in I} x_I - \prod_{ik \in I, jl \not \in I}
x_I + \prod_{il \in I, jk \not \in I} x_I \right \rangle : (\prod_{I \in
\mathcal I} x_I)^{\infty} \right), \\
\end{align*}
where the generating sets run over all $\{i,j,k,l\}$ with $1 \leq i
<j<k<l \leq n$.
\end{proof}

\begin{example}
\begin{enumerate}
\item The case $n=5$ is covered in Examples~\ref{e:torusquotient},
\ref{e:finalquadric}, and \ref{e:M05GIT}.

\item When $n=6$, the ideal $I_{2,6}$ is generated by the equations
$p_{ijkl}=x_{ij}x_{kl} - x_{ik}x_{jl} + x_{il}x_{jk}$, where $1 \leq i
< j < k <l \leq 6$ in the ring $\K[x_{ij} : 1 \leq i < j \leq 6]$.
The Cox ring of $X_{\Delta}$ is $\K[x_I : I \in \mathcal I]$.
Applying the change of coordinates given by the matrix $R$ the
$p_{ijkl}$ become $\widetilde{p}_{ijkl}=x_{ij}x_{kl}x_{ijm} x_{ijn} -
x_{ik}x_{jl}x_{ikm} x_{ikn} + x_{il}x_{jk}x_{ilm} x_{iln}$, where
$\{i,j,k,l,m,n\}=\{1,2,3,4,5,6\}$.  The ideal $I_{\overline{M}_{0,6}}$
has additional generators 
\begin{multline*}
q_{ij} = x_{ik} x_{jk} x_{ijk}^2 x_{lm}
x_{ln} x_{mn} - x_{il} x_{jl} x_{ijl}^2 x_{km} x_{kn} x_{mn}\\
 + x_{im}
x_{jm} x_{ijm}^2 x_{kl} x_{kn} x_{ln}
 - x_{in} x_{jn} x_{ijn}^2 x_{kl}
x_{km} x_{lm},
\end{multline*} 
for $1 \leq i < j \leq 6$, where again $\{i,j,k,l,m,n\}= \{1,\dots, 6\}$.
\end{enumerate}

\end{example}

\begin{remark}
When $n=4$ we can still follow the recipe of Section~\ref{s:Equations}
to obtain the equations for $\overline{M}_{0,4} \cong \mathbb P^1$
inside $X_{\Sigma^*} \cong \mathbb P^2$.  The Grassmannian $G(2,5)
\subseteq \mathbb P^5$ is the hypersurface
$Z(x_{12}x_{34}-x_{13}x_{24}+x_{14}x_{23})$.  In this case $M_{0,4} =
Z(x_{34}-z_{24}+1)$ in $T^2=\Spec(\K[z_{24}^{\pm 1},z_{34}^{\pm 1}])$.
Also \renewcommand{\arraystretch}{0.8}
\renewcommand{\arraycolsep}{2pt}
$$R=\left( \text{\footnotesize $\begin{array}{rrrrrr} 1 & 0 & -1 & -1 & 0 & 1 \\ 0 & 1 & -1
& -1 & 1 & 0 \\ \end{array}$} \right) \left( \text{\footnotesize $\begin{array}{rrr} 1 & 0 & 0 \\ 
0 & 1 & 0\\
0 & 0 & 1 \\
0 & 0 & 0 \\
0 & 0 & 0 \\
0 & 0 & 0 \\
0 & 0 & 0 \\
\end{array}$}
\right), \text{  so } 
\renewcommand{\arraystretch}{0.8}
\renewcommand{\arraycolsep}{2pt}
V=\left( \text{\footnotesize $\begin{array}{rrrrrr} 1 & 0 & 0 & 0 & 0 & 0 \\
0 & 1 & 0 & 0 & 0 & 0 \\
0 & 0 &1 & 0 & 0 & 0 \\
\end{array}$}
\right).
$$ Thus $\nu(x_{12}x_{34}-x_{13}x_{24}+x_{14}x_{23}) = y_1-y_2+y_3
\subseteq \K[y_1,y_2,y_3] = \Cox(\mathbb P^2)$.
\end{remark}

\section{VGIT and the effective cone of $\overline{M}_{0,n}$ }
\label{s:VGIT}

An important invariant of a projective variety $Y$ is its
pseudoeffective cone $\overline{\mathrm{Eff}}(Y)$, and one of the
primary goals of Mori theory is to understand the decomposition of
this cone into Mori chambers.  In this section we prove Theorem
\ref{t:GITM0n}, which is the second part of
Theorem~\ref{t:M0nmainthm}.  We identify a subcone $\mathcal G$ of the
effective cone of $\overline{M}_{0,n}$ for which one has that
$\overline{M}_{0,n}$ can be constructed as a GIT quotient of an affine
variety with linearization determined by a given $D \in \mathcal G$.
This prompts Question~\ref{q:MDregion}, which asks whether the cone
$E$ of divisors spanned by boundary classes is a Mori Dream region of
the effective cone.

Let $H=\Hom(\DivCl(X_{\Delta}), \K^{\times}) \cong
(\K^{\times})^{b+n}$.  The torus $H$ acts on $\mathbb A^{|\mathcal
I|}$ with weights given by the columns of the matrix $G$ of Equation~\ref{eqtn:G}.
Recall the cone $\mathcal G(X_{\Delta})= \bigcap_{\sigma \in \Delta}
\pos([D_{I}] : I \not \in \sigma)$ from Definition~\ref{d:Gcone}.  Let
$i \colon \MOn \rightarrow X_{\Delta}$ be the inclusion of
Theorem~\ref{t:definingThm}.  The pullback $i^*(\mathcal
G(X_{\Delta})) \subset N^1(\MOn)\otimes \mathbb R$ is a subcone of the
nef cone of $\MOn$.

\begin{theorem} \label{t:GITM0n}
\begin{enumerate}
\item   For rational $\alpha \in
\intt(\mathcal G(X_{\Delta}))$ we have the GIT construction of $\MOn$ as
$$\MOn = Z(I_{\MOn}) \git_{\alpha} H,$$ where $Z(I_{\MOn}) \subset
\mathbb A^{|\mathcal I|}$ is the affine subscheme defined by
$I_{\MOn}$.

\item Let $n \geq 5$.  Given $\beta \in \mathbb N A_n$ there is
$\alpha \in \mathbb NG$ for which
$$Z(I_{\MOn}) \git_{\alpha} H = G(2,n) \git_{\beta} T^{n-1},$$ so all
GIT quotients of $G(2,n)$ by $T^{n-1}$ can be obtained from $\MOn$ by
variation of GIT.
\end{enumerate}
\end{theorem}

\begin{proof}
The first part of the theorem is a direct application of
Proposition~\ref{p:GIT}.  For the second, note that for $n \geq 5$
there are no repeated columns in the matrix $D$, so the result follows
from Theorem~\ref{t:VGIT} and Corollary~\ref{c:VGIT}.
\end{proof}

\begin{remark}
\begin{enumerate}
\item The second part of the theorem is still true for $n=4$, as
$\overline{M}_{0,4} \cong \mathbb P^1$, which is also equal to one of
the GIT quotients of $G(2,4)$.
\item When $n=5$, $I_{\overline{M}_{0,5}} = I_{2,5}$, so we see that
$\overline{M}_{0,5}=G(2,5)\git_{\alpha} T^4$ for $\alpha \in
\intt(\mathcal G(X_{\Delta}))$.  In this case the second part of the
theorem is a tautology.
\item For larger $n$, since $A$ and $C$ are nonnegative matrices the
expression \linebreak $-\min \{(Cu)_i : Au=\beta, u \in \mathbb Q_{\geq 0}^{m+1}
\}$ of Theorem~\ref{t:VGIT} is a nonpositive integer, and thus we can choose $\alpha =
(\beta,0)$.
\end{enumerate}
\end{remark}

\begin{remark}
A natural problem is to give a combinatorial description of these equations for $\MOn$, similar to that given for the
GIT quotient $G(2,n)\git_{\beta} T^{n-1}$ in \cite{HMSV}.  Generators
for the corresponding ring can still be described by graphs on $n$
vertices, but an added complication is that their Kempe lemma is not
true; the generators corresponding to noncrossing graphs no longer
give a basis for the degree-one part of the ring.
\end{remark}

In \cite{HuKeel} Hu and Keel introduce Mori Dream spaces, which are
 varieties whose effective cones are polyhedral and for which the Mori
 chamber decomposition breaks this cone into a finite number of
 polyhedral pieces.
They prove that if a variety $Y$ is a Mori Dream Space, then there is
 an embedding of $Y$ into a projective toric variety $X_{\Sigma}$, and
 thus a GIT construction of $Y$, so that $Y$ and $X_{\Sigma}$ have
 isomorphic Picard groups and effective cones, and all small $\mathbb
 Q$-factorial modifications of $Y$ can be obtained by variation of GIT
 quotient from $Y$.  They also define the weaker notion of subcone $C
 \subseteq \overline{\mathrm{Eff}}(Y)$ being a Mori dream region.
 This holds if
 $$R=\bigoplus_{D \in C}H^0(V, D)$$ is finitely generated.  

 Hu and Keel raise the question of whether $\overline{M}_{0,n}$ is a
 Mori Dream space.  Theorem~\ref{t:definingThm} shows that $\MOn$
 embeds into $X_{\Delta}$ and Lemma~\ref{l:ClequalsPic} shows that
 $\Pic(\MOn) \cong \DivCl(X_{\Delta})$.  However, for $n \ge 6$, Keel
 and Vermeire showed that the the cone $E$ generated by the boundary
 divisors of $\MOn$ is a proper subcone of
 $\mathrm{Eff}(\overline{M}_{0,n})$ \cite{Vermeire}, so the effective
 cones of $\overline{M}_{0,n}$ and $X_{\Delta}$ differ.  On the other
 hand, Castravet~\cite{Castravet} showed $\overline{M}_{0,6}$ is
 indeed a Mori Dream space.  We believe that $E$ may be a Mori Dream
 region.  Indeed, the GIT chambers of $Z(I_{\MOn}) \git_{\alpha} H$
 divide $E$ into polyhedral chambers, each of which corresponds to a
 different compactification of $M_{0,n}$.  The chamber containing
 $\mathcal G(X_{\Delta})$ corresponds to the compactification $\MOn$.
 
\begin{question} \label{q:MDregion}
\begin{enumerate}
\item Let $\overline{M}_{0,6} \rightarrow X_{\Sigma}$ be the embedding
  of $\overline{M}_{0,6}$ into a toric variety of dimension $24$ with
  isomorphic Picard group and effective cone guaranteed by the Mori
  dream space construction.  Let $X_{\Sigma'}$ be the toric subvariety
  obtained from $X_{\Sigma}$ by removing $T^{24}$-orbits of $X_{\Sigma}$
  not intersecting $\overline{M}_{0,6}$.  Is $X_{\Sigma'}$ obtained
  from $X_{\Delta}$ in a natural way (such as by tropical
  modifications) that generalizes to $n>6$?

\item Is there a toric embedding $\MOn \rightarrow X_{\Sigma}$ with
$\MOn$ and $X_{\Sigma}$ having isomorphic Picard groups and effective cones
that can be obtained from $X_{\Delta}$ by tropical modifications?
This would support the conjecture that $\MOn$ is a Mori dream space.
\item 
Let $E$ be the closed subcone of $\mathrm{Eff}(\MOn)$ spanned by the
boundary divisors,
let $S$
be the Cox ring of $X_{\Delta}$ and $I_{\overline{M}_{0,n}}$
the ideal of $\overline{M}_{0,n}$ in $S$. Is 
 $$(S/{I}_{\overline{M}_{0,n}})_{D} \cong
 H^0(\overline{M}_{0,n}, D),$$ for all $D \in E$?  This would imply
 that $E$ is a Mori Dream region for $\MOn$.
\end{enumerate}
 \end{question}

\def\cprime{$'$} \def\cprime{$'$}

 
\end{document}